\documentclass[12pt]{amsart}
\usepackage{color}

\usepackage{enumerate}
\usepackage{graphics}

\usepackage{mathrsfs}
\usepackage{amssymb}

   \newcommand{\Cdb}{\mbox{$\mathbb{C}$}}

   \newcommand{\Ndb}{\mbox{$\mathbb{N}$}}

   \newcommand{\Rdb}{\mbox{$\mathbb{R}$}}

   \newcommand{\B}{\mbox{${\mathcal B}$}}

   \newcommand{\E}{\mbox{${\mathcal E}$}}
   \newcommand{\F}{\mbox{${\mathcal F}$}}
   
   \renewcommand{\H}{\mbox{${\mathcal H}$}}
   \newcommand{\I}{\mbox{${\mathcal I}$}}
   \newcommand{\J}{\mbox{${\mathcal J}$}}
   \newcommand{\K}{\mbox{${\mathcal K}$}}
   \newcommand{\Ll}{\mbox{${\mathcal L}$}}

   \renewcommand{\P}{\mbox{${\mathcal P}$}}

   \newcommand{\T}{\mbox{${\mathcal T}$}}

   \newcommand{\norm}[1]{\Vert#1\Vert}
   \newcommand{\bignorm}[1]{\bigl\Vert#1\bigr\Vert}
   \newcommand{\Bignorm}[1]{\Bigl\Vert#1\Bigr\Vert}

\newtheorem{theorem}{Theorem}
\newtheorem{lemma}[theorem]{Lemma}
\newtheorem{proposition}[theorem]{Proposition}
\newtheorem{corollary}[theorem]{Corollary}

\newtheorem{remark}[theorem]{Remark}

\theoremstyle{remark}

\newtheorem{example}[theorem]{Example}

\begin{document}

\title[When do triple operator integrals take value in the trace class?]
{When do triple operator integrals take value in the trace class?}

\author[C. Coine]{Clement Coine}
\email{clement.coine@univ-fcomte.fr}
\author[C. Le Merdy]{Christian Le Merdy}
\email{clemerdy@univ-fcomte.fr}
\author[F. Sukochev]{Fedor Sukochev}
\email{f.sukochev@unsw.edu.au}

\address{F.S. : School of Mathematics \& Statistics, University of NSW,
Kensington NSW 2052, AUSTRALIA}
\address{C.C., C.L. : 
Laboratoire de Mathématiques de Besan\c con, UMR 6623, CNRS, Universit\'e Bourgogne Franche-Comt\'e,
25030 Besan\c{c}on Cedex, FRANCE}

\date{\today}

\maketitle

\begin{abstract}
Consider three normal operators $A,B,C$ on separable Hilbert space $\H$ 
as well as scalar-valued spectral measures $\lambda_A$ on $\sigma(A)$,
$\lambda_B$ on $\sigma(B)$ and $\lambda_C$ on $\sigma(C)$. For any
$\phi\in L^\infty(\lambda_A\times \lambda_B\times \lambda_C)$ and any
$X,Y\in S^2(\H)$, the space of Hilbert-Schmidt operators on $\H$, we
provide a general definition of a triple operator integral 
$\Gamma^{A,B,C}(\phi)(X,Y)$ belonging to $S^2(\H)$ in such a way
that $\Gamma^{A,B,C}(\phi)$ belongs to the space $B_2(S^2(\H)\times
S^2(\H), S^2(\H))$ of bounded bilinear operators on $S^2(\H)$, and the
resulting mapping $\Gamma^{A,B,C}\colon L^\infty(\lambda_A\times \lambda_B\times \lambda_C)
\to B_2(S^2(\H)\times
S^2(\H), S^2(\H))$ is a $w^*$-continuous isometry. Then we show that a function
$\phi\in L^\infty(\lambda_A\times \lambda_B\times \lambda_C)$
has the property that $\Gamma^{A,B,C}(\phi)$ maps $S^2(\H)\times
S^2(\H)$ into $S^1(\H)$, the space of trace class operators on $\H$, if and only if it 
has the following factorization property: there exist a Hilbert space $H$ and
two functions $a\in L^{\infty}(\lambda_A \times \lambda_B ; H)$ and 
$b\in L^{\infty}(\lambda_B\times \lambda_C ; H)$
such that $\phi(t_1,t_2,t_3)= \left\langle a(t_1,t_2),b(t_2,t_3) \right\rangle$
for a.e. $(t_1,t_2,t_3) \in \sigma(A) \times \sigma(B) \times \sigma(C).$
This is a bilinear version of Peller's Theorem characterizing double operator 
integral mappings $S^1(\H)\to S^1(\H)$.
In passing we show that for any separable Banach spaces $E,F$,
any $w^*$-measurable esssentially bounded function 
valued in the Banach space $\Gamma_2(E,F^*)$ of operators from 
$E$ into $F^*$ factoring through Hilbert space admits a $w^*$-measurable
Hilbert space factorization.
\end{abstract}

\section{Introduction}
Let $\mathcal{H}$ be a separable Hilbert space. Let $S^2(\mathcal{H})$ 
denote the space of Hilbert-Schmidt operators on $\H$ and let $S^1(\mathcal{H})$
denote the space of trace class operators on $\H$. Let $A,B$ be two normal operators on $\H$.
Any bounded Borel function $\phi$ on $\sigma(A)\times\sigma(B)$ gives rise to a
double operator integral mapping $\Gamma^{A,B}(\phi)\colon S^2(\mathcal{H})\to S^2(\mathcal{H})$
formally defined as
$$
\Gamma^{A,B}(\phi)(X) =\int_{\sigma(A)\times\sigma(B)} \phi(s,t)\,
\text{d}E^{A}(s)\, X\, \text{d}E^{B}(t),\qquad
X\in S^2(\mathcal{H}),
$$
where $E^{A}$ and $E^B$ denote the spectral measures of $A$ and $B$, respectively.
Double operator integrals were initially defined by Daletskii and Krein \cite{DK} and then
dramatically developed in a series 
of papers of Birman-Solomiak \cite{BS1,BS2,BS3}. They play a prominent role 
in various aspects of operator theory, especially in the perturbation theory. 
We refer the reader to the survey papers \cite{BS4,Peller2015}  and
to the book \cite{ST} for a large volume of information on
this topic and its applications.

In \cite{Peller1985}, V.V. Peller gave a characterization of double operator integral mappings
which restrict to a bounded operator on $S^1(\H)$. He showed that $\Gamma^{A,B}(\phi)$ is a bounded operator
from $S^1(\H)$ into itself if and only there exist a Hilbert space $H$ and two functions
$a\in L^\infty(E^{A};H)$ and $b\in L^\infty(E^{B};H)$ such that
$$
\phi(s,t) = \langle a(s), b(t)\rangle \qquad a.e.\hbox{-}(s,t).
$$
This property means that the operator $L^1(E^{A})\to L^{\infty}(E^{B})$ with 
kernel $\phi$ factors through Hilbert space. 
We refer to \cite{Peller1985} and \cite{Hiai} for other equivalent formulations.

The purpose of this paper is to study an analogue 
of Peller's Theorem for triple operator 
integrals. This issue was motivated by a recent work 
of the authors together with D. Potapov and A. Tomskova
on perturbation theory  \cite{CLPST1}. In this paper 
the construction of triple operator 
integral mappings which do not map $S^2(\H)\times 
S^2(\H)$ into $S^1(\H)$ played a fundamental role;
see also \cite{CLPST2} and \cite{PSST} for related work.

The paper \cite{CLPST1} contains the following result on infinite matrices (see Theorems 1, 7 and
Corollary 8 in the latter paper). Let 
$M=\{m_{ikj}\}_{i,k,j\geq 1}$ be a three-dimensional matrix
with entries in $\mathbb C$. Let $(E_{ij})_{i,j\geq 1}$ denote the standard matrix units.
Then the bilinear Schur multiplier $B_M$ formally defined by
$$
B_M(X,Y):=\sum_{i,j,k\ge 1} m_{ikj}x_{ik}y_{kj}\,E_{ij}, \quad
X=\{x_{ij}\}_{i,j\ge 1},\ Y=\{y_{ij}\}_{i,j\ge 1},
$$
defines a bounded bilinear operator from
$S^2\times  S^2$ into $S^1$
if and only if there exist a Hilbert space $H$
and two bounded families
$(a_{ik})_{i,k\geq 1}$ and $(b_{jk})_{j,k\geq 1}$ in $H$ such that
%\begin{equation}\label{grothen}
$$
m_{ikj} = \langle a_{ik}, b_{jk}\rangle,
\qquad i,k,j\geq 1.
$$
%\end{equation}

Triple operator integral mappings can be regarded as (far reaching) extensions of 
bilinear Schur multipliers, hence the above result serves as a guide for our investigation.
In Section 3 we revisit an old construction of Pavlov \cite{Pav} providing 
a general definition of triple operator integral mappings 
$$
\Gamma^{A,B,C}(\phi) \colon S^2(\H)\times S^2(\H)\longrightarrow S^2(\H),
$$
where $A,B,C$ are normal operators on $\H$, $\lambda_A$, $\lambda_B$,
$\lambda_C$  are scalar valued spectral measures on the spectra $\sigma(A)$,
$\sigma(B)$, $\sigma(C)$, respectively,
and $\phi\in L^{\infty}(\lambda_A \times \lambda_B \times \lambda_C)$. 
We show in Theorem \ref{GammaABC} and Corollary \ref{Iso}
that $\Gamma^{A,B,C}$ is an isometry from 
$L^{\infty}(\lambda_A \times \lambda_B \times \lambda_C)$ into
$B_2(S^2(\H)\times S^2(\H), S^2(\H))$, the space of bounded bilinear
maps from $S^2(\H)\times S^2(\H)$ into $S^2(\H)$, and that $\Gamma^{A,B,C}$ 
is $w^*$-continuous (i.e. continuous in the $w^*$-topologies of the dual 
spaces $L^{\infty}(\lambda_A \times \lambda_B \times \lambda_C)$ and
$B_2(S^2(\H)\times S^2(\H), S^2(\H))$).

Our main result, established in Section 6 (see Theorem \ref{main}), asserts that
$$
\Gamma^{A,B,C}(\phi) : S^2(\mathcal{H}) \times S^2(\mathcal{H}) \longrightarrow S^1(\mathcal{H})
$$
if and only if there exist a Hilbert space $H$ and two functions
$$
a\in L^{\infty}(\lambda_A \times \lambda_B ; H) \qquad
\text{and} \qquad
b\in L^{\infty}(\lambda_B\times \lambda_C ; H)
$$
such that 
\begin{equation}\label{Factor}
\phi(t_1,t_2,t_3)= \left\langle a(t_1,t_2),b(t_2,t_3) \right\rangle,\qquad a.e.\hbox{-}(t_1,t_2,t_3).
\end{equation}
In Section \ref{add}, we recover Peller's Theorem as a special case of the above statement and we
compare our triple operator integrals with previous constructions.

Multiple operator integrals are a very active topic at the moment. 
 In addition to the already
mentioned papers  \cite{CLPST1, CLPST2}, we refer the reader to 
\cite{ANP1, ANP2, AP, ACDS, C, CLSS, LS, Peller2006, PSS-SSF} for important results,
as well as to \cite{ST} and the references therein.

The proof of Theorem \ref{main} combines several techniques and intermediate 
results which are discussed in Sections 2-5. First, the $w^*$-continuity
of $\Gamma^{A,B,C}$ plays a crucial role as it allows to reduce 
various computations to tensor product
manipulations. The relevant background on tensor products and duality is
provided in Section \ref{prel}. 
Second, in order to study the 
factorization property (\ref{Factor}), which is about functions 
only, we need to develop triple operator
integrals associated with functions, in parallel with the construction
of $\Gamma^{A,B,C}$. This is achieved in Subsection \ref{Functions}. The 
link between the two constructions, which is fundamental for our purpose, 
is given in Subsection \ref{Op-to-Fu} (see Proposition \ref{Connection}).
Third, $w^*$-measurable versions of vector-valued 
$L^p$-spaces and Hilbert space factorizations appear naturally in our investigation.
Sections \ref{Lp} and \ref{factorization} are devoted
to these two topics. Our main result, of independent interest, is the following. 
Let $E,F$ be separable Banach spaces
and let $(\Omega,\mu)$ be a separable measure space. Let 
$\Gamma_2(E, F^*)$ be the space of all bounded linear operators $E \rightarrow F^*$ 
which factor through Hilbert space. This is a dual space (see (\ref{dualGamma2})). 
We show that if $\phi\colon\Omega\to 
\Gamma_2(E, F^*)$ is a $w^*$-measurable essentially bounded function, then there exist
a separable Hilbert space $H$ and two $w^*$-measurable essentially bounded functions
$\alpha\colon\Omega\to B(E,H)$ and $\beta\colon\Omega\to B(F,H)$ such that
$$
\bigl\langle\bigl[\phi(t)\bigr] (x),y\bigr\rangle\,=\, \bigl\langle
\bigl[\alpha(t)\bigr] (x), \bigl[\beta(t)\bigr] (y)\bigr\rangle
$$
almost everywhere, for any $x\in E$ and $y\in F$.

\bigskip
We end this Introduction with a few notations and conventions.
Throughout the paper we will use the notation
$\norm{\ }_p$ for the norms on various $L^p$-spaces, which may
be either classical ones or vector valued ones. The notations
$\norm{\ }_1$ and $\norm{\ }_2$ will also be used on the
spaces of trace class operators and Hilbert-Schmidt operators, 
respectively (see Subsection \ref{Hilbert}).

Whenever $\Sigma$ is a set and 
$V\subset\Sigma$ is a subset we let $\chi_V\colon\Sigma\to \{0,1\}$ 
denote the characteristic function of $V$.

The Hilbertian direct sum of any family $(\H_i)_{i\in I}$ of Hilbert spaces
will be denoted by
$$
\overset{2}{\oplus}_{i\in I} \H_i.
$$
Likewise, the notation $\H\overset{2}{\oplus}\K$ will stand for the 
Hilbertian direct sum of any two Hilbert spaces $\H$ and $\K$.

Whenever $E,F$ are two Banach spaces, a bounded linear map 
$u\colon E^*\to F^*$ will be called $w^*$-continuous when 
it is continuous with respect to the $w^*$-topologies of 
$E^*$ and $F^*$. This is equivalent to the fact that $u$ is the adjoint 
of a bounded linear map from $F$ into $E$.
 We recall that when a $w^*$-continuous map $u\colon E^*\to F^*$ is an isometry, then
its range is $w^*$-closed, and $u$ induces a $w^*$-homeomorphism between
$E^*$ and its range. The latter is therefore a dual space
and $u\colon E^*\to u(E^*)$ is an isometric $w^*$-homeomorphic identification
between the dual spaces $E^*$ and $u(E^*)$.

\vskip 1cm

\section{Preliminaries and background}\label{prel}

\subsection{Normal operators and scalar-valued spectral measures}\label{SV}

We assume that the reader is familiar with the general spectral theory of normal operators on
Hilbert space, for which we refer e.g. to \cite[Chapters 12 and 13]{Rudin} and \cite[Sections 14 and 15]{Conway}.
Let $\mathcal{H}$ be a separable Hilbert space and
let $A$ be a (possibly unbounded) normal operator on $\mathcal H$. We let $\sigma(A)$ denote the spectrum of $A$ and 
we let $E^A$ denote the spectral measure of $A$, defined on the Borel subsets of $\sigma(A)$.

By definition a scalar-valued spectral measure for $A$ is a positive finite measure $\lambda_A$
on the Borel subsets of $\sigma(A)$, such that $\lambda_A$ and $E^A$ have the same sets of measure zero. 
Such measures exist, thanks to the separability assumption on $\mathcal{H}$. Indeed let
$$
W^*(A)\subset B(\mathcal{H})
$$
be the von Neumann algebra generated by the range of $E^A$. By \cite[Corollary 14.6]{Conway},
$W^*(A)$ has a separating vector $e$. It follows that  
%\begin{equation}\label{lambdaA}
$$
\lambda_A := \|E^A(.)e\|^2
$$
%\end{equation}
is a scalar-valued spectral measure for $A$. (This construction is given
in \cite[Section 15]{Conway} for a bounded $A$.)

The Borel functional calculus for $A$ takes any bounded Borel   
function $f\colon \sigma(A)\to\Cdb$ to the bounded operator
$$
f(A):=\int_{\sigma(A)} f(t) \ \text{d}E^A(t)\,.
$$
According to \cite[Theorem 15.10]{Conway}, it induces a $w^*$-continuous (=normal)
$*$-representation
\begin{equation}\label{piA}
\pi_A \colon L^{\infty}(\lambda_A) \longrightarrow B(\mathcal{H}),
\end{equation}
As a matter of fact, the space $L^{\infty}(\lambda_A)$ does not depend on the choice 
of the scalar-valued spectral measure $\lambda_A$.  Without ambiguity,
we may write 
$f(A)=\pi_A(f)$ for any $f\in L^{\infty}(\lambda_A)$.

\subsection{Tensor products and duality}\label{tensorproducts}

We give a brief summary of tensor product formulas to be used in the sequel.
Let $E$, $F$ and $G$ be Banach spaces. We let $B(E,G)$ be the Banach space of 
all bounded linear operators
from $E$ into $G$. Then we let $B_2(E\times F,G)$ be the Banach space of 
all bounded bilinear operators
$T\colon E\times F\to G$, equipped with
$$
\norm{T}=\sup\bigl\{\norm{T(x,y)}\, :\, x\in E,\, y\in F,\, \norm{x}\leq 1,\, \norm{y}\leq 1\bigr\}.
$$

If $z\in E \otimes F$, the projective tensor norm of $z$ is defined by
$$
\|z\|_{\wedge} := \inf \left\lbrace \sum \|x_i\| \|y_i\| \right\rbrace,
$$
where the infimum runs over all finite families $(x_i)_i$ in $E$ and $(y_i)_i$ in $F$ such 
that 
$$
z=\sum_i x_i\otimes y_i. 
$$
The completion $E \overset{\wedge}{\otimes} F$ of $(E\otimes F,\norm{\ }_{\wedge})$
is called the projective tensor product of $E$ and $F$.

To any $T\in B_2(E\times F, G)$, one can associate 
a linear map $\widetilde{T}\colon E\otimes F\to G$
by the formula
$$
\widetilde{T}(x\otimes y)=T(x,y),\qquad  x\in E,\, y\in F.
$$
Then $\widetilde{T}$ is bounded on $(E\otimes F,\norm{\ }_{\wedge})$, 
with $\norm{\widetilde{T}}=\norm{T}$, and 
hence the mapping $T\mapsto
\widetilde{T}$ gives rise to an isometric identification
\begin{equation}\label{biliproj}
B_2(E\times F, G) = B(E \overset{\wedge}{\otimes} F,G).
\end{equation}

In the case $G=\Cdb$, this implies that the mapping taking any 
functional $\omega\colon E\otimes F\to\Cdb$
to the operator $u\colon E\to F^*$ defined by $\langle u(x),
y\rangle=\omega(x\otimes y)$ for any $x\in E, y\in F$,
induces an isometric identification
\begin{equation}\label{dualproj}
(E \overset{\wedge}{\otimes} F)^*=B(E,F^*).
\end{equation}
We refer to \cite[Chapter 8, Theorem 1 $\&$ Corollary 2]{Diestel} for these classical facts.

Let $(\Omega,\mu)$ be a $\sigma$-finite measure space and let $L^1(\Omega;F)$ denote the
Bochner space of integrable functions from $\Omega$ into $F$.
By \cite[Chapter 8, Example 10]{Diestel}, the natural embedding $L^1(\Omega)
\otimes F\subset L^1(\Omega;F)$
extends to an isometric isomorphism
\begin{equation}\label{L1tensor}
L^1(\Omega;F) = L^1(\Omega)\overset{\wedge}{\otimes}F.
\end{equation}
By (\ref{dualproj}), this implies
\begin{equation}\label{L1tensorcor}
L^1(\Omega;F)^*=B(L^1(\Omega),F^*).
\end{equation}

Let $E,W$ be Banach spaces.
We say that an operator $u \colon 
E \rightarrow W$ factors through a Hilbert space if there exist a Hilbert space $H$ and 
two operators $\alpha: E \rightarrow H$ and $\beta : H \rightarrow W$ such that $u=\beta\alpha$.
We denote by $\Gamma_2 (E,W)$ the space of all such operators. For any $u \in \Gamma_2 (E,W)$, define
$$
\gamma_2(u)=\inf\bigl\{\norm{\alpha}\norm{\beta}\bigr\},
$$
where the infimum runs over all factorizations of $u$ as above.
Then $\gamma_2$ is a norm on $\Gamma_{2}(E,W)$ and the latter is a 
Banach space, see e.g. \cite{DiestelJar}
or \cite[Chapter 2]{PisierCBMS}.

We will make crucial use of the fact that if $W$ is a dual space, 
then $\Gamma_2 (E,W)$ is a dual space as well. 
Indeed assume that $W=F^*$ for some 
Banach space $F$. Then there exists a norm 
$\gamma_2^*\leq \norm{\ }_\wedge$ on $E\otimes F$
such that if we let $E\hat{\otimes}_{\gamma_2^*} F$ 
denote the completion of
$(E\otimes F,\gamma_2^*)$, then (\ref{dualproj}) induces an isometric identification
\begin{equation}\label{dualGamma2}
(E\hat{\otimes}_{\gamma_2^*} F)^* = \Gamma_{2}(E,F^*).
\end{equation}
See e.g. \cite[Theorem 5.3]{PisierBook} for a definition 
of $\gamma_2^*$ (that we will not use here) and a proof.
By construction, the canonical embedding
$\Gamma_{2}(E,F^*)\to B(E,F^*)$ is $w^*$-continuous.

\subsection{Operators on Hilbert spaces and trace duality}\label{Hilbert}

Let $\H,\K$ be Hilbert spaces and let ${\rm tr}$ be the trace on $B(\mathcal{K})$.
We let $S^1(\mathcal{K},\mathcal{H})$ denote the space of trace class operators $T\colon \mathcal{K}\to 
\mathcal{H}$, equipped with $\norm{T}_1={\rm tr}(\vert T\vert)$, where $\vert T\vert =(T^*T)^{\frac12}$. 
We recall that the pairing
$$
\langle S,T\rangle ={\rm tr}(ST),\qquad T\in S^1(\mathcal{K},\mathcal{H}),\ 
S\in B(\mathcal{H},\mathcal{K}),
$$
induces an isometric identification
\begin{equation}\label{Dual-S1}
B(\mathcal{H},\mathcal{K}) =  S^1(\mathcal{K},\mathcal{H})^*.
\end{equation}
Let $S^2(\mathcal{K},\mathcal{H})$ denote the space of Hilbert-Schmidt operators 
$T\colon \mathcal{K}\to 
\mathcal{H}$, equipped with $\norm{T}_2=\bigl({\rm tr}(\vert T\vert^2)\bigr)^{\frac12}$.
Then the above duality pairing also yields an isometric identification
\begin{equation}\label{dualS2}
S^2(\mathcal{H},\mathcal{K}) =  S^2(\mathcal{K},\mathcal{H})^*.
\end{equation}

Given any two Banach spaces $E,G$, it is customary to identify 
$E^*\otimes G$ with the space of bounded finite rank operators from
$E$ into $G$. Indeed for any $x^*\in E^*$ and $g\in G$, 
$x^*\otimes g$ is identified with the element of $B(E,G)$
taking any $x\in E$ to $x^*(x)g$. 
We apply this principle to Hilbert spaces. We let $\overline{\K}$
denote the complex conjugate of $\K$ and recall 
the canonical identification $\K^*=\overline{\K}$. Then we regard
$\overline{\mathcal{K}}\otimes \mathcal{H}$
as the space of finite rank operators
from $\mathcal{K}$ into $\mathcal{H}$.
In this identification, for any $\eta\in\mathcal{K}$ and $\xi\in\mathcal{H}$, 
$\overline{\eta}\otimes \xi\colon \mathcal{K}
\to \mathcal{H}$ denotes the operator taking any $z\in 
\mathcal{K}$ to $\langle z,\eta\rangle \xi$.

We recall that $\overline{\mathcal{K}}\otimes \mathcal{H}$ is both a
dense subspace of $S^1(\mathcal{K},\mathcal{H})$ and $S^2(\mathcal{K},\mathcal{H})$.

\subsection{Measurable Schur multipliers}\label{Schurmulti}

Let $(\Omega_1, \mu_1)$ and $(\Omega_2, \mu_2)$ be two $\sigma$-finite measure spaces. 
If $J\in L^2(\Omega_1 \times \Omega_2)$, the operator
\begin{equation*}
\begin{array}[t]{lccc}
X_J : & L^2(\Omega_1) & \longrightarrow & L^2(\Omega_2) \\
& r & \longmapsto & \displaystyle \int_{\Omega_1} J(t,\cdotp)r(t)\, \text{d}\mu_1(t)  \end{array}
\end{equation*}
is a Hilbert-Schmidt operator and $\|X_J\|_2=\|J\|_2$. Further
any element of $S^2(L^2(\Omega_1), L^2(\Omega_2))$ has this form (see e.g.
\cite[Thm VI. 23]{RS}). We summarize
these facts by writing an isometric identification
\begin{equation}\label{S2=L2}
L^2(\Omega_1 \times \Omega_2) = S^2(L^2(\Omega_1), L^2(\Omega_2)).
\end{equation}

Let $\psi\in L^{\infty}(\Omega_1\times \Omega_2)$. Thanks to the above identity,
we may associate the operator
\begin{equation*}
\begin{array}[t]{lccc}
R_{\psi} : & S^2(L^2(\Omega_1), L^2(\Omega_2)) & \longrightarrow & S^2(L^2(\Omega_1), L^2(\Omega_2)) \\
& X_J & \longmapsto & X_{\psi J}  \end{array}
\end{equation*}
whose norm is equal to $\norm{\psi}_\infty$.
We say that $\psi$ is a measurable Schur multiplier if $R_{\psi}$ extends to a bounded operator
(still denoted by)
$$
R_{\psi} \colon \mathcal{K}(L^2(\Omega_1), L^2(\Omega_2)) 
\longrightarrow B(L^2(\Omega_1), L^2(\Omega_2)),
$$
where $\mathcal{K}(L^2(\Omega_1), L^2(\Omega_2))$ denotes the space of compact operators from
$L^2(\Omega_1)$ into $L^2(\Omega_2)$. The density of Hilbert-Schmidt operators in compact operators ensures 
that this extension is necessarily unique.

For any $\psi\in L^{\infty}(\Omega_1\times \Omega_2)$, one may define 
$u_{\psi} \in B(L^1(\Omega_1), L^{\infty}(\Omega_2))$ by
$$
u_{\psi}(r)=\int_{\Omega_1} \psi(t,\cdotp)r(t) \,\text{d}\mu_1(t),\qquad r\in L^1(\Omega_1).
$$
Applying (\ref{L1tensorcor}) with $F=L^1(\Omega_2)$ together with the identity 
$L^1(\Omega_1;L^1(\Omega_2))=L^1(\Omega_1\times\Omega_2)$, we obtain an  
isometric $w^*$-homeomorphic identification
\begin{equation}\label{IntForm}
L^{\infty}(\Omega_1\times \Omega_2) = B(L^1(\Omega_1), L^{\infty}(\Omega_2)).
\end{equation}
A thorough look at this identification reveals that it is given by the mapping
$\psi\mapsto u_\psi$. Thus we have 
$$
\norm{u_\psi}=\norm{\psi}_\infty
$$
and any element of $B(L^1(\Omega_1), L^{\infty}(\Omega_2))$ 
is an operator $u_\psi$ for some (unique) $\psi$.

The first part of Theorem \ref{Pellerthm} below
is a remarkable characterization of measurable Schur multipliers. In the 
discrete case it was stated by Pisier in \cite[Theorem 5.1]{PisierBook} who refers himself 
to some earlier work of Grothendieck. 
For the general case considered here we refer to Haagerup
\cite{Ha} and Spronk \cite[Section 3.2]{Spronk}.
Peller's characterization of double operator integral mappings which restrict
to a bounded operator $S^1(\H)\to S^1(H)$ is
closely related to this factorization result. Indeed, 
Theorem \ref{Pellerthm} (1) below is implicit in \cite{Peller1985}.

For the second part of the next result, recall 
that by (\ref{dualGamma2}) and 
(\ref{dualproj}), 
$$
\Gamma_2 (L^1 (\Omega_1), L^{\infty}(\Omega_2))
\qquad\hbox{and}\qquad
B\bigl(\mathcal{K}(L^2(\Omega_1), 
L^2(\Omega_2)), B(L^2(\Omega_1), L^2(\Omega_2))\bigr)
$$
are both dual spaces.

\begin{theorem}\label{Pellerthm} \

\begin{itemize}
\item [(1)]  \cite{Ha, Peller1985, PisierBook, Spronk} A function $\psi\in L^{\infty}(\Omega_1\times \Omega_2)$
is a measurable Schur multiplier if and only if the operator 
$u_{\psi}$ belongs to $\Gamma_2 (L^1 (\Omega_1), L^{\infty}(\Omega_2))$, and we have
$$
\gamma_2(u_{\psi})=\| R_{\psi} \|
$$
in this case.
\item [(2)]
Moreover the isometric embedding 
$$
\Gamma_2 (L^1 (\Omega_1), L^{\infty}(\Omega_2))\hookrightarrow
B\bigl(\mathcal{K}(L^2(\Omega_1), L^2(\Omega_2)), B(L^2(\Omega_1), L^2(\Omega_2))\bigr)
$$
taking any $u_{\psi}\in \Gamma_2 (L^1 (\Omega_1), L^{\infty}(\Omega_2))$ to $R_\psi$
is $w^*$-continuous.
\end{itemize}
\end{theorem}

\begin{proof}
Let us prove (2).
Let $\psi\in L^{\infty}(\Omega_1\times \Omega_2)$ and 
let $(\psi_\iota)_\iota$ be a net of $L^{\infty}(\Omega_1\times \Omega_2)$ such that $u_\psi$ and 
the operators $u_{\psi_\iota}$ belong to $\Gamma_2 (L^1 (\Omega_1), L^{\infty}(\Omega_2))$ for any $\iota$,
$(u_{\psi_\iota})_\iota$ is a bounded net in the latter space,
and $u_{\psi_\iota}\to u_\psi$ in the $w^*$-topology of 
$\Gamma_2 (L^1 (\Omega_1), L^{\infty}(\Omega_2))$.
This implies that $u_{\psi_\iota}\to u_\psi$ in the $w^*$-topology of $B(L^1 (\Omega_1), 
L^{\infty}(\Omega_2))$ (see the comments following (\ref{dualGamma2})).
According to (\ref{IntForm}), this means that $\psi_\iota\to \psi$
in the $w^*$-topology of $L^\infty(\Omega_1\times\Omega_2)$.

Let $\xi,\xi'\in L^2(\Omega_1)$ and $\eta,\eta'\in L^2(\Omega_2)$.
For any $\iota$, $R_{\psi_\iota}(\overline{\xi}\otimes \eta)$ 
is the Hilbert-Schmidt operator associated to the 
$L^2$-function $\psi_\iota(\overline{\xi}\otimes \eta)$, hence
$$
\bigl\langle
\bigr[R_{\psi_\iota}(\overline{\xi}\otimes \eta)\bigr](\xi'),\eta'\bigr\rangle
=\int_{\Omega_1\times\Omega_2} \psi_\iota(t_1,t_2) \overline{\xi(t_1)} \xi'(t_1) 
\eta(t_2)\overline{\eta'(t_2)}
\,\text{d}\mu_1(t_1)\text{d}\mu_2(t_2)\,.
$$
The right-hand side of this equality is the action of $\psi_\iota\in L^\infty(\Omega_1\times\Omega_2)$
on the $L^1$-function 
$$
(t_1,t_2)\mapsto \overline{\xi(t_1)} \xi'(t_1) \eta(t_2)\overline{\eta'(t_2)}.
$$
Since $\psi=w^*$-$\lim_\iota\psi_\iota$, this implies that 
$$
\bigl\langle
\bigr[R_{\psi_\iota}(\overline{\xi}\otimes \eta)\bigr](\xi'),\eta'\bigr\rangle
\longrightarrow \bigl\langle
\bigr[R_{\psi}(\overline{\xi}\otimes \eta)\bigr](\xi'),\eta'\bigr\rangle.
$$
By linearity, this implies that for any finite rank operator $\sigma\colon 
L^2(\Omega_1)\to L^2(\Omega_2)$, 
$R_{\psi_\iota}(\sigma)\to R_{\psi}(\sigma)$ is the weak operator topology of 
$B(L^2(\Omega_1), L^2(\Omega_2))$. 
Since $(u_{\psi_\iota})_\iota$ is a bounded net, $(R_{\psi_\iota})_\iota$ is bounded as well.
By the density of finite rank operators 
in $\mathcal{K}(L^2(\Omega_1), L^2(\Omega_2))$, we deduce that 
for any $\sigma$ in the latter space,
$R_{\psi_\iota}(\sigma)\to R_{\psi}(\sigma)$ is the weak operator topology  of 
$B(L^2(\Omega_1), L^2(\Omega_2))$. 
Using again the boundedness of $(R_{\psi_\iota})_\iota$, we deduce that 
$R_{\psi_\iota}(\sigma)\to R_{\psi}(\sigma)$ in the $w^*$-topology of 
$B(L^2(\Omega_1), L^2(\Omega_2)\bigr)$ for any $\sigma\in
\mathcal{K}(L^2(\Omega_1), L^2(\Omega_2))$ and finally
that $R_{\psi_\iota}\to R_{\psi}$ in the 
$w^*$-topology of $B\bigl(\mathcal{K}(L^2(\Omega_1), L^2(\Omega_2)), 
B(L^2(\Omega_1), L^2(\Omega_2)\bigr)$.
\end{proof}

\vskip 1cm

\section{Triple operator integral mappings}\label{triple}

Multiple operator integrals appeared in many recent papers with 
various definitions, see in particular \cite{ANP1, ANP2, AP, ACDS, Peller2006, 
PSS-SSF}. In this section we provide a definition of triple operator
integrals associated to a triple $(A,B,C)$ of normal operators on $\H$,
based on the construction of a natural
$w^*$-continuous mapping from 
$L^\infty(\lambda_A\times\lambda_B\times\lambda_C)$ into 
$B_2(S^2(\H)\times S^2(\H),S^2(\H))$,
see Theorem \ref{GammaABC}. We will show in Corollary \ref{Iso} that this 
mapping is actually an isometry.
Further the construction extends to multiple operator integrals, see Proposition
\ref{nintegrals}. It turns out that this construction is equivalent to an
old definition of multiple operator integrals due to Pavlov \cite{Pav};
this will be explained in Remark \ref{Rem-Pavlov}. 

In Subsection \ref{Functions}, we give an analogue of the construction
for functions, in the spirit of Subsection \ref{Schurmulti}. Finally in
Subsection \ref{Op-to-Fu}, we establish a fruitful connection between 
triple operator integrals associated with operators and 
triple operator integrals associated with functions.

\subsection{Triple operator integrals associated with operators}\label{Operators}

Let $\mathcal{H}$ be a separable Hilbert space and let $A,B,C$ be 
(possibly unbounded) normal operators on $\mathcal{H}$. 
Denote by $E^A, E^B$ and $E^C$ their spectral measures and 
let $\lambda_A, \lambda_B$ and $\lambda_C$ be scalar-valued 
spectral measures for $A$, $B$ and $C$ (see Subsection \ref{SV}).

Let $\mathcal{E}_1 \subset L^{\infty}(\lambda_A)$, 
$\mathcal{E}_2 \subset L^{\infty}(\lambda_B)$ and $\mathcal{E}_3 \subset L^{\infty}(\lambda_C)$ be  
the spaces of simple functions on $(\sigma(A),\lambda_A)$, $(\sigma(B),\lambda_B)$
and $(\sigma(C),\lambda_C)$, respectively. We let  
$$
\Gamma \colon \mathcal{E}_1 \otimes \mathcal{E}_2 \otimes \mathcal{E}_3 
\longrightarrow B_2(S^2(\mathcal{H}) \times S^2(\mathcal{H}),
S^2(\mathcal{H}))
$$
be the unique linear map such that
\begin{equation}\label{Def-G}
\Gamma(f_1\otimes f_2\otimes f_3)(X,Y)=f_1(A)Xf_2(B)Yf_3(C)
\end{equation}
for any $f_1\in \mathcal{E}_1$, $f_2\in \mathcal{E}_2$ and $f_3\in \mathcal{E}_3$, and 
for any $X,Y \in S^2(\mathcal{H})$.

\begin{lemma}\label{Gamma}
For all $\phi \in \mathcal{E}_1 \otimes \mathcal{E}_2 \otimes \mathcal{E}_3$, and for all 
$X,Y \in S^2(\mathcal{H})$, we have
$$
\norm{\Gamma(\phi)(X,Y)}_2\leq\norm{\phi}_\infty\norm{X}_2\norm{Y}_2.
$$
\end{lemma}

\begin{proof}
Let $\phi \in \mathcal{E}_1 \otimes \mathcal{E}_2 \otimes \mathcal{E}_3$. 
There exists a finite family $(F_i^1)_i$ (respectively $(F_j^2)_j$ and $(F_k^3)_k$)
of pairwise disjoint measurable 
subsets of $\sigma(A)$ (respectively of $\sigma(B)$ and $\sigma(C)$) of positive measures, 
as well as 
a family $(m_{ijk})_{i,j,k}$ of complex numbers such that 
\begin{equation}\label{phi}
\phi = \sum_{i,j,k} m_{ijk}\,\chi_{F_i^1} \otimes \chi_{F_j^2} \otimes \chi_{F_k^3}.
\end{equation}
Then we have 
\begin{equation}\label{Norm-phi}
\|\phi\|_{\infty}=\sup_{i,j,k}|m_{ijk}|.
\end{equation}
Let $X,Y \in \mathcal{S}^2(\mathcal{H})$. 
According to the definition of $\Gamma$, we have
$$
\Gamma(\phi)(X,Y) =
\sum_{i,j,k}  m_{ijk} E^A(F^1_i)XE^B(F^2_j)YE^C(F^3_k).
$$
By the pairwise disjointnesses of $(F_i^1)_i$ and $(F_k^3)_k$, 
the elements 
$$
\Bigl(\sum_j m_{ijk}E^A(F^1_i)XE^B(F^2_j)YE^C(F^3_k)\Bigr)_{i,k}
$$
are pairwise orthogonal in $S^2(\mathcal{H})$. Hence 
$$
\| \Gamma(\phi)(X,Y) \|_2^2
= \sum_{i,k} \Bignorm{\sum_j m_{ijk} E^A(F^1_i)XE^B(F^2_j)YE^C(F^3_k)}^2_2.
$$
Applying the Cauchy-Schwarz inequality and (\ref{Norm-phi}),
we deduce that 
\begin{align*}
\| \Gamma(\phi)(X,Y) \|_2^2
& \leq \|\phi\|_{\infty}^2 \sum_{i,k} \Bigl(\sum_j 
\bignorm{E^A(F^1_i)XE^B(F^2_j)}_2\bignorm{E^B(F^2_j) YE^C(F^3_k)}_2 \Bigr)^2 \\
& \leq \|\phi\|_{\infty}^2 \sum_{i,k} \Bigl( 
\sum_j \bignorm{E^A(F^1_i)XE^B(F^2_j)}^2_2 \Bigr) \Bigl(\sum_j 
\bignorm{E^B(F^2_j) YE^C(F^3_k)}_2^2 \Bigr) \\
& \leq \|\phi\|_{\infty}^2 \Bigl(\sum_{i,j} \bignorm{E^A(F^1_i)XE^B(F^2_j)}^2_2 \Bigr) 
\Bigl(\sum_{j,k} \bignorm{E^B(F^2_j) YE^C(F^3_k)}^2_2 \Bigr).
\end{align*}
Since the elements $E^A(F^1_i)XE^B(F^2_j)$ are pairwise orthogonal in $S^2(\mathcal{H})$ we have
\begin{align*}
\sum_{i,j} \bignorm{E^A(F^1_i)XE^B(F^2_j)}^2_2
& = \Bignorm{\sum_{i,j} E^A(F^1_i)XE^B(F^2_j)}_2^2 \\
& = \bignorm{E^A\left( \cup_i F_i^1 \right)XE^B\left( \cup_j F_j^2 \right)}_2^2 \\
& \leq \|X\|_2^2.
\end{align*}
Similarly, 
$$
\sum_{j,k} \bignorm{E^B(F^2_j) YE^C(F^3_k)}^2_2 \leq \|Y\|_2^2.
$$
This yields the result.
\end{proof}

We let 
$$
G:=\overline{\mathcal{E}_1 \otimes \mathcal{E}_2 \otimes \mathcal{E}_3}^{\|.\|_{\infty}} 
\subset L^{\infty}(\lambda_A \times \lambda_B \times \lambda_C),
$$
equipped with the $L^\infty$-norm, 
and we let $\tau\colon L^1(\lambda_A \times \lambda_B \times \lambda_C)\to G^*$
be the canonical map defined by 
$$
\bigl\langle \tau(\varphi),\phi\bigr\rangle = \int_{\sigma(A)\times\sigma(B)\times\sigma(C)}
\varphi\phi\  \text{d}(\lambda_A\times\lambda_B\times\lambda_C)\,,
\qquad \varphi\in L^1, \ \phi\in G.
$$
This is obviously a contraction.

We claim that $\tau$ is actually an isometry. To check this fact,
consider $\varphi\in
\mathcal{E}_1 \otimes \mathcal{E}_2 \otimes \mathcal{E}_3$, that we write
as a finite sum
$$
\varphi = \sum_{i,j,k} c_{ijk}\,\chi_{F_i^1} \otimes \chi_{F_j^2} \otimes \chi_{F_k^3},
$$
with $c_{ijk}\in\Cdb\setminus\{0\}$ and $(F_i^1)_i$ (respectively $(F_j^2)_j$ and $(F_k^3)_k$)
being pairwise disjoint measurable 
subsets of $\sigma(A)$ (respectively of $\sigma(B)$ and $\sigma(C)$), with positive measures.
Then
$$
\norm{\varphi}_1 =  \sum_{i,j,k} \vert c_{ijk}\vert\,\lambda_A(F_i^1)
\lambda_B(F_j^2)\lambda_C(F_k^3).
$$
Let $\phi$ be defined by (\ref{phi}), with $m_{ijk}=\vert c_{ijk}\vert c_{ijk}^{-1}$. Then
$\norm{\phi}_\infty=1$ by (\ref{Norm-phi}) and 
$$
\bigl\langle\tau(\varphi),\phi\bigr\rangle = \sum_{i,j,k} m_{ijk} c_{ijk} \lambda_A(F_i^1)
\lambda_B(F_j^2)\lambda_C(F_k^3) = \norm{\varphi}_1.
$$
Hence we have $\norm{\tau(\varphi)}=\norm{\varphi}_1$ as expected. Since 
$\mathcal{E}_1 \otimes \mathcal{E}_2 \otimes \mathcal{E}_3$ is dense in
$L^1(\lambda_A \times \lambda_B \times \lambda_C)$, this implies that
$\tau$ is an isometry.

According to this property, we now consider 
$L^1(\lambda_A \times \lambda_B \times \lambda_C)$
as a subspace of $G^*$.

By (\ref{biliproj}),  (\ref{dualproj}) and (\ref{dualS2}), we have isometric identifications
\begin{align*}
B_2(S^2(\mathcal{H}) \times S^2(\mathcal{H}), S^2(\mathcal{H}))
& =B(S^2(\mathcal{H}) \overset{\wedge}{\otimes} S^2(\mathcal{H}),S^2(\mathcal{H})) \\
& = \bigl(S^2(\mathcal{H}) \overset{\wedge}{\otimes} 
S^2(\mathcal{H}) \overset{\wedge}{\otimes} S^2(\mathcal{H})\bigr)^*.
\end{align*}
It is easy to check that the duality pairing providing this identification reads
$$
\bigl\langle T, X\otimes Y\otimes Z\bigr\rangle =\text{tr}\bigl(T(X,Y)Z\bigr)
$$
for any $T\in B_2(S^2(\mathcal{H}) \times S^2(\mathcal{H}), S^2(\mathcal{H}))$
and any $X,Y,Z\in S^2(\mathcal{H})$.

We set 
$$
E=S^2(\mathcal{H}) \overset{\wedge}{\otimes} 
S^2(\mathcal{H}) \overset{\wedge}{\otimes} S^2(\mathcal{H}).
$$
According to Lemma \ref{Gamma}, $\Gamma$ uniquely extends to a contraction
$$
\widetilde{\Gamma}\colon G\longrightarrow   
B_2(S^2(\mathcal{H}) \times S^2(\mathcal{H}), S^2(\mathcal{H})) = E^*.
$$
We can therefore consider 
$S=\widetilde{\Gamma}^*_{|E} : E \rightarrow G^*$, 
the restriction of $\widetilde{\Gamma}^*$ to $E\subset E^{**}$.

\begin{lemma}\label{eta}
The operator $S$ takes its values in the subspace 
$L^1(\lambda_A \times \lambda_B \times \lambda_C)$ of $G^*$.
\end{lemma}

\begin{proof}
Let $\P=\overline{\mathcal{H}}\otimes \mathcal{H} 
\otimes \overline{\mathcal{H}}\otimes \mathcal{H}\otimes 
\overline{\mathcal{H}}\otimes \mathcal{H}$. Recall that we identify
$\overline{\mathcal{H}}\otimes\mathcal{H}$ with the space of finite rank
operators on $\mathcal{H}$. Then 
$\overline{\mathcal{H}}\otimes\mathcal{H}$ is a dense 
subspace of $\mathcal{S}^2(\mathcal{H})$. Consequently 
$\P$ is a dense subspace of $E$. Since $S$ is continuous, it therefore 
suffices to show that $S(\P) \subset L^1(\lambda_A 
\times \lambda_B \times \lambda_C)$.
Consider $\eta_1,\eta_2,\eta_3,\xi_1,\xi_2,\xi_3$ in $\mathcal{H}$ and 
$\omega=\overline{\xi_1}\otimes \eta_1\otimes \overline{\xi_2}\otimes 
\eta_2\otimes \overline{\xi_3}\otimes \eta_3$.
Such elements span  $\P$ hence it suffices to
check that $S(\omega)$ belongs to $L^1(\lambda_A \times \lambda_B \times \lambda_C)$.
Let $f_1\in \mathcal{E}_1$, $f_2\in \mathcal{E}_2$ and $f_3\in \mathcal{E}_3$. We have
\begin{align*}
\left\langle S(\omega),f_1\otimes f_2 \otimes f_3 \right\rangle
& = \left\langle \omega, \Gamma(f_1\otimes f_2 \otimes f_3) \right\rangle \\
& = \text{tr}\Bigl(\bigl[\Gamma(f_1\otimes f_2 \otimes f_3)
(\overline{\xi_1}\otimes \eta_1, \overline{\xi_2}\otimes \eta_2)\bigr]
(\overline{\xi_3}\otimes \eta_3)\Bigr) \\
& = \text{tr}\bigl(f_1(A)(\overline{\xi_1}\otimes \eta_1)f_2(B)
(\overline{\xi_2}\otimes \eta_2)f_3(C)
(\overline{\xi_3}\otimes \eta_3)\bigr) \\
& = \text{tr}\bigl((\overline{\xi_1}\otimes f_1(A)\eta_1)
(\overline{\xi_2}\otimes f_2(B)\eta_2)
(\overline{\xi_3}\otimes f_3(C)\eta_3)\bigr) \\
& = \text{tr}\bigl((\overline{\xi_3}\otimes f_1(A)\eta_1)
\left\langle f_3(C)\eta_3, \xi_2 \right\rangle \left\langle f_2(B)\eta_2, \xi_1 \right\rangle\bigr) \\
& = \left\langle f_3(C)\eta_3, \xi_2 \right\rangle 
\left\langle f_2(B)\eta_2, \xi_1 \right\rangle \left\langle f_1(A)\eta_1, \xi_3 \right\rangle.
\end{align*}
We mentioned that the functional calculus
$*$-representation $\pi_A\colon L^\infty(\lambda_A)\to B(\mathcal{H})$
from (\ref{piA}) is $w^*$-continuous. Thus $\pi_A$ is the adjoint
of some $w_A \colon S^1(\H)\to L^1(\lambda_A)$. Let $h_1=w_A(\overline{\xi_3}\otimes \eta_1)$.
Then this element of $L^1(\lambda_A)$ (which does not depend on $f_1$) satisfies
$$
\left\langle f_1(A)\eta_1, \xi_3 
\right\rangle= \int_{\sigma(A)} f_1 h_1\, \text{d}\lambda_A\,.
$$
(A thorough look at the
construction of $\pi_A$ shows that $h_1$ is actually
the Radon-Nikodym derivative of the measure $dE^A_{\eta_1,\xi_3}$ with respect to $\lambda_A$.)

Similarly, there exist $h_{2} \in L^1(\lambda_B)$ and $h_3 \in L^1(\lambda_C)$ not
depending on $f_2$ and $f_3$ such that
$\left\langle f_2(B)\eta_2, \xi_1 
\right\rangle= \int_{\sigma(B)} f_2 h_2\, \text{d}\lambda_B\,$ and 
$\left\langle f_3(C)\eta_3, \xi_2 
\right\rangle= \int_{\sigma(C)} f_3 h_3\, \text{d}\lambda_C\,$. Consequently,
$$
\left\langle S(\omega), f_1\otimes f_2 \otimes f_3 \right\rangle
= \int_{\sigma(A)\times\sigma(B)\times\sigma(C)}
(f_1\otimes f_2\otimes f_3)(h_1\otimes h_2\otimes h_3)\, 
\text{d}(\lambda_A\times \lambda_B\times\lambda_C).
$$
Since $\mathcal{E}_1\otimes \mathcal{E}_2\otimes\mathcal{E}_3$ is dense in $G$, this implies that 
$$
S(\omega) = h_1 \otimes h_{2}\otimes h_{3} \in L^1(\lambda_A \times \lambda_B \times \lambda_C).
$$
\end{proof}

\begin{theorem}\label{GammaABC}
There exists a unique $w^*$-continuous contraction
$$
\Gamma^{A,B,C} \colon L^{\infty}(\lambda_A \times \lambda_B \times \lambda_C) 
\longrightarrow B_2(S^2(\mathcal{H}) \times S^2(\mathcal{H}),S^2(\mathcal{H})),
$$
such that for any $f_1\in L^\infty(\lambda_A)$, $f_2 \in 
L^\infty(\lambda_B)$ and $f_3\in L^\infty(\lambda_C)$, and for any $X,Y \in S^2(\mathcal{H})$, we have
\begin{equation}\label{Def-G2}
\Gamma^{A,B,C}(f_1\otimes f_2\otimes f_3)(X,Y)=f_1(A)Xf_2(B)Yf_3(C).
\end{equation}
\end{theorem}

\begin{proof}
The uniqueness follows from the $w^*$-density of $L^\infty(\lambda_A)\otimes
L^\infty(\lambda_B)\otimes L^\infty(\lambda_C)$ in the dual space
$L^{\infty}(\lambda_A \times \lambda_B \times \lambda_C)$.

Lemma \ref{eta} provides a contraction $S\colon E\to L^1(\lambda_A \times \lambda_B \times \lambda_C)$.
Then its adjoint $S^*$ is a contraction from 
$L^{\infty}(\lambda_A \times \lambda_B \times \lambda_C)$
into $E^*=B_2(S^2(\mathcal{H}) \times S^2(\mathcal{H}),S^2(\mathcal{H}))$.
We set 
$$
\Gamma^{A,B,C}=S^*.
$$
By construction, $\Gamma^{A,B,C}$ is $w^*$-continuous and extends the map $\Gamma$ 
defined by (\ref{Def-G}). Property (\ref{Def-G2}) follows from (\ref{Def-G}) by 
$w^*$-continuity.
\end{proof}

Later on in Corollary \ref{Iso}, we will show that $\Gamma^{A,B,C}$ is 
actually an isometry.

Bilinear maps of the form $\Gamma^{A,B,C}(\phi)$ will be  called {\it 
triple operator integral mappings} in this paper. Operators of the form
$\Gamma^{A,B,C}(\phi)(X,Y)\colon\H\to\H$ are called triple operator integrals.
As indicated in the Introduction, our goal is to determine the functions 
$\phi\in L^{\infty}(\lambda_A \times \lambda_B \times \lambda_C)$ for 
which the triple operator integral mapping
$\Gamma^{A,B,C}(\phi)$ maps
$S^2(\mathcal{H}) \times S^2(\mathcal{H})$
into $S^1(\mathcal{H})$.

By similar computations (left to the reader), the above construction can be extended to 
$(n-1)$-tuple operator integrals, for any $n\geq 2$. 
One obtains the following statement, in which 
$B_{n-1}(S^2(\mathcal{H}) 
\times S^2(\mathcal{H}) \times \cdots \times S^2(\mathcal{H}),S^2(\mathcal{H}))$ 
denotes the space of bounded $(n-1)$-linear maps from the product of
$(n-1)$ copies of $S^2(\mathcal{H})$ taking values in $S^2(\mathcal{H})$.

\begin{proposition}\label{nintegrals} Let $n \geq 2$ and let 
$A_1, A_2, \ldots, A_n$ be normal operators on $\mathcal{H}$. For
any $i=1,\ldots,n$, let $\lambda_{A_i}$ be a
scalar-valued spectral measure for $A_i$ and let 
$\mathcal{E}_i\subset L^\infty(\lambda_{A_i})$ be the space of
simple functions on $(\sigma(A_i),\lambda_{A_i})$.
There exists a unique $w^*$-continuous contraction
$$
\Gamma^{A_1,A_2, \ldots, A_n} : L^{\infty}\left(\prod_{i=1}^n 
\lambda_{A_i}\right) \longrightarrow 
B_{n-1}(S^2(\mathcal{H}) \times S^2(\mathcal{H}) 
\times \cdots \times S^2(\mathcal{H}) 
\rightarrow S^2(\mathcal{H})),
$$
such that for any  $f_i \in L^\infty(\lambda_{A_i})$ and for any $X_1, 
\ldots, X_{n-1} \in S^2(\mathcal{H})$,
we have
\begin{align*}
\Gamma^{A_1,A_2, \ldots, A_n}(f_1\otimes\cdots\otimes f_n)
& (X_1,\ldots, X_{n-1})=\\
& f_1(A_1)X_1f_2(A_2) \cdots f_{n-1}(A_{n-1})X_{n-1}f_n(A_n).
\end{align*}
\end{proposition}

\begin{remark}\label{n=2} In the case $n=2$, the above proposition 
boils down to the original construction of double operator integrals
by Birman-Solomyak. Namely, let $A,B$ be two normal operators on $\H$,
and let 
$$
\Gamma^{A,B}\colon L^\infty(\lambda_A\times\lambda_B)\longrightarrow B(S^2(\H))
$$
be given by Proposition \ref{nintegrals}. For any $\phi\in L^\infty(\lambda_A\times\lambda_B)$,
let $\J(\phi)\colon S^2(\H)\to S^2(\H)$ be the operator constructed
in \cite[Section 3.1]{BS4} for the spectral measures associated with $A$ and $B$.
Then $\Gamma^{A,B}(\phi)$ coincides $\J(\phi)$.

We note for further use that $\Gamma^{A,B}$ is a $*$-representation 
of the von Neumann algebra $L^\infty(\lambda_A\times\lambda_B)$ on 
the Hilbert space $S^2(\H)$. This is easy to deduce from our definitions; 
also, this follows from \cite[(3.6) and (3.7)]{BS4}.
\end{remark}

\begin{remark}\label{Rem-Pavlov} As indicated in the introduction of this section,
the above construction turns out to be equivalent to Pavlov's definition of multiple
operator integrals given in \cite{Pav}. 
Let us briefly review Pavlov's construction from \cite{Pav},
and explain this `equivalence'. In this remark, 
we use terminology and references from \cite[Chapter 1]{Diestel}.

Let $n \geq 2$ and consider normal operators $A_1, A_2, \ldots, A_n$ as in Proposition
\ref{nintegrals}. Fix operators $X_1, \ldots, X_{n-1}$ in $S^2(\mathcal{H})$. 
Let $\Omega:=\sigma(A)\times\sigma(A_2)\times \cdots \times \sigma(A_n)$ and consider the set 
$\mathcal{F}$ consisting of finite unions of subsets of $\Omega$ of the form
$$
\Delta = F_1 \times F_2 \times \cdots \times F_n,
$$
where, for any $1\leq i \leq n, F_i$ is a Borel subset of $\sigma(A_i)$.

There exists a (necessarily unique) finitely additive vector measure 
$m\colon \mathcal{F}\to S^2(\H)$ such that 
\begin{equation}\label{mDelta}
m(\Delta)=E^{A_1}(F_1)X_1E^{A_2}(F_2) \cdots E^{A_{n-1}}(F_{n-1}) X_{n-1}E^{A_n}(F_n)
\end{equation}
for any $\Delta$ as above.

Pavlov first shows that $m$ is a measure of bounded semivariation and then proves that
$m$ is actually countably additive (see \cite[Theorem 1]{Pav}). Let $\mathcal{T}$
be the $\sigma$-field generated by $\F$. Since $S^2(\H)$ is reflexive, it follows from
\cite[Chapter 1, Section 5, Theorem 2]{Diestel} that $m$ has a (necessarily unique)
countably additive extension $\widetilde{m}\colon \T\to S^2(\H)$. Moreover $\widetilde{m}$
is a measure of bounded semivariation. Then using the fact that for all $i$,
$\lambda_{A_i}$ is a scalar-valued spectral measure for $A_i$, one can show that
$$
\widetilde{m} \ll \lambda_{A_1} \times \lambda_{A_2} \times \cdots \times \lambda_{A_n}
$$
on $\F$. This implies that $L^{\infty}(\lambda_{A_1} \times \lambda_{A_2} 
\times \cdots \times \lambda_{A_n})
\subset L^\infty(\widetilde{m})$ and hence, for any $\phi\in
L^{\infty}(\lambda_{A_1} \times \lambda_{A_2} \times \cdots \times 
\lambda_{A_n})$, one may define an integral
$$
\int_\Omega \phi(t)\,\text{d}\widetilde{m}(t)\ \in S^2(\H).
$$
See \cite[Chapter 1, Section 1, Theorem 13]{Diestel} for details. This element
is defined in \cite{Pav} as the multiple operator integral associated
to $\phi$ and $(X_1, \ldots, X_{n-1})$.

We claim that this construction 
is equivalent to the one given in the present paper, namely
$$
\int_\Omega \phi(t)\,\text{d}\widetilde{m}(t) = \Gamma^{A_1,A_2, 
\ldots, A_n}(\phi) (X_1, \ldots, X_{n-1}).
$$
To check this identity, let $w_1,w_2\colon L^{\infty}
(\lambda_{A_1} \times \lambda_{A_2} \times \cdots \times 
\lambda_{A_n})\to S^2(\H)$ be defined by 
$w_1(\phi) =\int_\Omega \phi(t)\,\text{d}\widetilde{m}(t)$ and 
$w_2(\phi)= \Gamma^{A_1,A_2, \ldots, A_n}(\phi) (X_1, \ldots, X_{n-1})$. 
For any $Z\in S^2(\H)$, the functional of $L^{\infty}
(\lambda_{A_1} \times \lambda_{A_2} \times \cdots \times \lambda_{A_n})$ taking $\phi$ to 
$\Bigl\langle\int_\Omega \phi(t)\,\text{d}\widetilde{m}(t)\, 
,Z\Bigr\rangle$ induces a countably additive measure on $\T$, which is absolutely
continuous with respect to $\lambda_{A_1} \times \lambda_{A_2} \times \cdots \times \lambda_{A_n}$. By the Radon-Nikodym Theorem
it is represented by an element of $L^1(\lambda_{A_1} \times \lambda_{A_2} \times \cdots \times \lambda_{A_n})$. Hence $w_1^*$ maps $S^2(\H)$ into
$L^1(\lambda_{A_1} \times \lambda_{A_2} \times \cdots \times \lambda_{A_n})$.
This implies that $w_1$ is $w^*$-continuous.
We know that $w_2$ is  
$w^*$-continuous as well, by Proposition \ref{nintegrals}. Further 
it is easy to derive from (\ref{mDelta}) that $w_1$ and $w_2$ coincide on $\E_1\otimes\cdots\otimes\E_n$.
These properties imply the equality $w_1=w_2$ as claimed.
\end{remark}

\subsection{Triple operator integrals associated with functions}\label{Functions}

Let $(\Omega_1, \mu_1), (\Omega_2, \mu_2)$ and $(\Omega_3, \mu_3)$ be three 
$\sigma$-finite measure spaces, 
and let $\phi \in L^{\infty}(\Omega_1 \times \Omega_2 \times \Omega_3)$. 
For any $J\in L^2(\Omega_1 \times \Omega_2)$ and $K\in L^2(\Omega_2\times \Omega_3)$, the function
$$
\Lambda(\phi)(K,J) \colon (t_1,t_3) \mapsto \int_{\Omega_2} 
\phi(t_1,t_2,t_3)J(t_1,t_2)K(t_2,t_3) \,\text{d}\mu_2(t_2)
$$
is a well-defined element of $L^2(\Omega_1 \times \Omega_3)$ 
with $L^2$-norm less than $\|\phi\|_{\infty}\|J\|_2\|K\|_2$. 
Indeed, by the Cauchy-Schwarz inequality, we have
\begin{align*}
\int_{\Omega_1 \times \Omega_3}
& \left(\int_{\Omega_2} |\phi(t_1,t_2,t_3) J(t_1,t_2)K(t_2,t_3)| 
\text{d}\mu_2(t_2) \right)^2 \text{d}\mu_1(t_1)\text{d}\mu_3(t_3) \\
& \leq \|\phi\|_{\infty}^2 \int_{\Omega_1 \times \Omega_3} 
\left(\int_{\Omega_2} |J(t_1,t_2)K(t_2,t_3)| \text{d}\mu_2(t_2) 
\right)^2 \text{d}\mu_1(t_1)\text{d}\mu_3(t_3) \\
& \leq \|\phi\|_{\infty}^2 \int_{\Omega_1 \times \Omega_3} 
\left(\int_{\Omega_2} |J(t_1,t_2)|^2 \text{d}\mu_2(t_2) \right) 
\left(\int_{\Omega_2} |K(t_2,t_3)|^2 \text{d}\mu_2(t_2) \right) 
\text{d}\mu_1(t_1)\text{d}\mu_3(t_3) \\
& \leq \|\phi\|_{\infty}^2 \left(\int_{\Omega_1 \times \Omega_2} |J(t_1,t_2)|^2  
\text{d}\mu_1(t_1) \text{d}\mu_2(t_2) \right) 
\left(\int_{\Omega_2\times \Omega_3} 
|K(t_2,t_3)|^2 \text{d}\mu_2(t_2)\text{d}\mu_3(t_3) \right).
\end{align*}
Thus $\Lambda(\phi)$ is a bounded bilinear map from 
$L^2(\Omega_2 \times \Omega_3)\times L^2(\Omega_1 \times \Omega_2)$
into $L^2(\Omega_1 \times \Omega_3)$.
By the isometric identification between $L^2(\Omega_1 \times \Omega_2)$ 
and $S^2(L^2(\Omega_1), L^2(\Omega_2))$ given by (\ref{S2=L2}), and their
analogues for $(\Omega_2,\Omega_3)$ and $(\Omega_1,\Omega_3)$, we may consider
that we actually have a bounded  bilinear map
$$
\Lambda(\phi) \colon  S^2(L^2(\Omega_2), L^2(\Omega_3)) \times 
S^2(L^2(\Omega_1), L^2(\Omega_2))\longrightarrow S^2(L^2(\Omega_1), L^2(\Omega_3)).
$$
In Section \ref{mainresult} we will characterize the functions $\phi$ 
for which $\Lambda(\phi)$ maps $S^2(L^2(\Omega_2), L^2(\Omega_3)) \times 
S^2(L^2(\Omega_1), L^2(\Omega_2))$ into the trace class
$S^1(L^2(\Omega_1), L^2(\Omega_3))$.

Let $E(\Omega_1,\Omega_2,\Omega_3)
= S^2(L^2(\Omega_2), L^2(\Omega_3)) \overset{\wedge}{\otimes} S^2(L^2(\Omega_1), 
L^2(\Omega_2)) \overset{\wedge}{\otimes} S^2(L^2(\Omega_3), L^2(\Omega_1))$.
Arguing as in the preceding subsection, we obtain an isometric identification
$$
E(\Omega_1,\Omega_2,\Omega_3)^*=
B_2( S^2(L^2(\Omega_2), L^2(\Omega_3)) \times 
S^2(L^2(\Omega_1), L^2(\Omega_2)), S^2(L^2(\Omega_1), L^2(\Omega_3)))
$$
for the duality pairing given by
$$
\bigl\langle T,  Y\otimes X\otimes Z\bigr\rangle =\text{tr}\bigl(T(Y,X)Z\bigr)
$$
for any bounded bilinear $T\colon S^2(L^2(\Omega_2), L^2(\Omega_3)) \times 
S^2(L^2(\Omega_1), L^2(\Omega_2))\to S^2(L^2(\Omega_1), L^2(\Omega_3))$ and 
for any 
$X\in S^2(L^2(\Omega_1), L^2(\Omega_2))$, $Y\in S^2(L^2(\Omega_2), 
L^2(\Omega_3))$ and $Z\in S^2(L^2(\Omega_3), L^2(\Omega_1))$.

The following is an analogue of Theorem \ref{GammaABC} for the present setting.

\begin{proposition}\label{Function-case}
The mapping
$$
\Lambda\colon L^\infty(\Omega_1\times\Omega_2\times\Omega_3)\longrightarrow 
E(\Omega_1,\Omega_2,\Omega_3)^*
$$
defined above is a $w^*$-continuous isometry.
\end{proposition}

\begin{proof}
Write $E=E(\Omega_1,\Omega_2,\Omega_3)$ for brevity. 
Consider three functions
$J\in L^2(\Omega_1 \times \Omega_2)$, $K\in L^2(\Omega_2\times \Omega_3)$
and $L\in L^2(\Omega_3\times \Omega_1)$. It follows from the computation at the beginning of the
present subsection that 
$$
(t_1,t_3)\mapsto \int_{\Omega_2}\vert J(t_1,t_2)\vert
\vert K(t_2,t_3)\vert\, \text{d}\mu_2(t_2)
$$
is square integrable. Consequently,
the function
$$
\varphi\colon (t_1,t_2,t_3)\mapsto J(t_1,t_2)K(t_2,t_3)L(t_3,t_1)
$$
belongs to $L^1(\Omega_1\times\Omega_2\times\Omega_3)$. 
Further if
$X_J\in S^2(L^2(\Omega_1), L^2(\Omega_2))$, 
$Y_K\in S^2(L^2(\Omega_2), L^2(\Omega_3))$ 
and $Z_L\in S^2(L^2(\Omega_3), L^2(\Omega_1))$
denote the Hilbert-Schmidt operators associated with $J$, $K$ and $L$, respectively, then it follows 
from above that
$$
\bigl\langle\Lambda(\phi),  Y_K\otimes X_J\otimes Z_L\bigr\rangle_{E^*,E}
\,=\, 
\int_{\Omega_1\times\Omega_2\times\Omega_3}
\phi\varphi\,\text{d}(\mu_1\otimes\mu_2\otimes\mu_3) 
=\,\langle\phi,\varphi\rangle_{L^\infty,L^1}
$$
for any $\phi\in  L^\infty(\Omega_1\times\Omega_2\times\Omega_3)$.
This readily implies that $\Lambda$ is $w^*$-continuous.

We already showed that $\Lambda$ is a contraction, let us now prove that it is 
an isometry. Let $\phi\in L^\infty(\Omega_1\times\Omega_2\times\Omega_3)$,
with $\norm{\phi}_\infty>1$. 
We aim at showing that 
$\norm{\Lambda(\phi)}_{E^*}>1$. There exist a function 
$\varphi\in L^1 (\Omega_1\times\Omega_2\times\Omega_3)$ such that 
$\norm{\varphi}_1=1$ and $\langle \phi,\varphi\rangle_{L^\infty,L^1}>1$.
By the density of simple functions in $L^1$, we may assume that
$$
\varphi = \sum_{i,j,k} m_{ijk}\,\chi_{F_i^1} \otimes \chi_{F_j^2} \otimes \chi_{F_k^3},
$$
where $(F_i^1)_i$ (respectively $(F_j^2)_j$ and $(F_k^3)_k$)
is a finite family of pairwise disjoint measurable 
subsets of $\Omega_1$ (respectively of $\Omega_2$ and $\Omega_3$)
and $m_{ijk}\in\Cdb$ for any $i,j,k$. Let $\psi\in E$ be defined by
$$
\psi = \sum_{i,j,k} m_{ijk}\bigl(\chi_{F_j^2} \otimes \chi_{F_k^3}\bigr)
\otimes \bigl(\chi_{F_i^1} \otimes \chi_{F_j^2}\bigr)
\otimes \bigl(\chi_{F_k^3} \otimes \chi_{F_i^1}\bigr).
$$
For any $i,j,k$, we have
\begin{align*}
\bigl\langle \Lambda(\phi),   \bigl(\chi_{F_j^2} \otimes \chi_{F_k^3}\bigr) &
\otimes \bigl(\chi_{F_i^1} \otimes \chi_{F_j^2}\bigr)
\otimes \bigl(\chi_{F_k^3} \otimes \chi_{F_i^1}\bigr)\bigr\rangle_{E^*,E} \\
& =\int_{\Omega_1\times\Omega_2\times \Omega_3} 
\phi(t_1,t_2,t_3) \chi_{F_i^1}(t_1)\chi_{F_j^2}(t_2)\chi_{F_k^3}(t_3)\,
\text{d}\mu_1(t_1)\text{d}\mu_2(t_2)\text{d}\mu_3(t_3)\,.
\end{align*}
This implies that
$$
\langle \Lambda(\phi),\psi\rangle_{E^*,E}\,=\,\langle \phi,\varphi\rangle_{L^\infty,L^1},
$$
and hence that $\langle \Lambda(\phi),\psi\rangle_{E^*,E}>1$.
Now observe that by the definition of the projective tensor product 
(see Subsection \ref{tensorproducts}), we have
$$
\norm{\psi}_E\leq\sum_{i,j,k}\vert m_{ijk}\vert
\norm{\chi_{F_i^1} \otimes \chi_{F_j^2}}_2
\norm{\chi_{F_j^2} \otimes \chi_{F_k^3}}_2
\norm{\chi_{F_k^3} \otimes \chi_{F_j^2}}_2.
$$
Moreover, 
$$
\norm{\chi_{F_i^1} \otimes \chi_{F_j^2}}_2 = 
\norm{\chi_{F_i^1}}_2 \norm{\chi_{F_j^2}}_2 = 
\lambda_1(F_i^1)^{\frac12}\lambda_2(F_j^2)^{\frac12}.
$$
Likewise, $\norm{\chi_{F_j^2} \otimes \chi_{F_k^3}}_2 = 
\lambda_2(F_j^2)^{\frac12}\lambda_3(F_k^3)^{\frac12}$
and $\norm{\chi_{F_k^3} \otimes \chi_{F_j^2}}_2 = 
\lambda_3(F_k^3)^{\frac12}\lambda_1(F_i^1)^{\frac12}$. We deduce that 
$$
\norm{\psi}_E\leq\sum_{i,j,k}\vert m_{ijk}\vert 
\lambda_1(F_j^2)\lambda_2(F_j^2)\lambda_3(F_k^3).
$$
The right-hand side of this inequality is nothing but the 
$L^1$-norm of $\varphi$. Thus we have proved that
$\norm{\psi}_E \leq \norm{\varphi}_1=1$. 
This implies that $\norm{\Lambda(\phi)}_{E^*}>1$ as expected.
\end{proof}

\subsection{Passing from operators to functions}\label{Op-to-Fu}

Let $\mathcal{H}$ be a separable Hilbert space and let $A,B$ and $C$ be 
normal operators on $\mathcal{H}$. We keep the notations from Subsection \ref{Operators}.
We associate the three measure spaces 
$$
(\Omega_1,\mu_1)=(\sigma(C),\lambda_C),\qquad
(\Omega_2,\mu_2)=(\sigma(B),\lambda_B)\qquad\hbox{and}\qquad
(\Omega_3,\mu_3)=(\sigma(A),\lambda_A)
$$ 
and consider the mapping $\Lambda$
defined in Subsection \ref{Functions} for these three measure spaces. 
It maps $L^\infty(\lambda_A\times\lambda_B\times\lambda_C)$ into 
$$
B_2(S^2(L^2(\lambda_B), L^2(\lambda_A))\times S^2(L^2(\lambda_C), L^2(\lambda_B)), 
S^2(L^2(\lambda_C),L^2(\lambda_A))).
$$
The main purpose of this subsection is to establish a precise connection between
this mapping $\Lambda$ and the triple operator integral mapping
$\Gamma^{A,B,C}$ from Theorem \ref{GammaABC}.

We may suppose that
$$
\lambda_A(.)=\|E^A(.)e_1\|^2, \quad \lambda_B(.)=
\|E^B(.)e_2\|^2 \quad \text{and} \quad \lambda_C(.)=\|E^C(.)e_3\|^2
$$
for some separating vectors
$e_1,e_2,e_3\in\mathcal{H}$ (see Subsection \ref{SV}).

There exists a (necessarily unique) 
linear map
$\rho_A\colon \mathcal{E}_1 \longrightarrow   \mathcal{H}$
satisfying
$$
\rho_A(\chi_F) = E^A(F)e_1
$$
for any Borel set $F\subset\sigma(A)$. For any finite family $(F_i)_i$
of pairwise disjoint measurable subsets of $\sigma(A)$ and for
any family $(\alpha_i)_i$ of complex numbers, we have
\begin{align*}
\Bignorm{\rho_A\Bigl(\sum_i\alpha_i \chi_{F_i}\bigr)}^2 & 
=\Bignorm{\sum_i\alpha_i E^A(F_i)e_1}^2 \\ &
= \sum_i \vert\alpha_i\vert^2\norm{E^A(F_i)e_1}^2\\ &
= \sum_i \vert\alpha_i\vert^2\lambda_A(F_i)\\ &
=\Bignorm{\sum_i\alpha_i \chi_{F_i}}_2^2.
\end{align*}
Hence $\rho_A$ extends to an isometry (still denoted by) 
$$
\rho_A : L^2(\lambda_A) \longrightarrow \mathcal{H}.
$$ 
Denote by $\mathcal{H}_A$ the range of $\rho_A$. We obtain
$$
L^2(\lambda_A) \overset{\rho_A}{\equiv} \mathcal{H_A}.
$$
We similarly define $\rho_B, \rho_C$ and $\mathcal{H}_B, \mathcal{H}_C \subset \mathcal{H}$ 
such that
$$
L^2(\lambda_B) \overset{\rho_B}{\equiv} \mathcal{H_B} \quad\text{and} 
\quad L^2(\lambda_C) \overset{\rho_C}{\equiv} \mathcal{H_C}.
$$

We may consider $S^2(\mathcal{H}_B,\mathcal{H}_A)$ as a subspace of $S^2(\mathcal{H})$ in a natural way.
Namely we write $\mathcal{H}=\mathcal{H}_B\overset{2}\oplus \mathcal{H}_B^\perp$ and 
$\mathcal{H}=\mathcal{H}_A\overset{2}\oplus \mathcal{H}_A^\perp$ and identify 
any $S\in S^2(\mathcal{H}_B,\mathcal{H}_A)$ with the matrix
$$
\begin{pmatrix} S & 0 \\ 0 & 0\end{pmatrix}\ \in \,S^2\bigl(\mathcal{H}_B\overset{2}\oplus \mathcal{H}_B^\perp,
\mathcal{H}_A\overset{2}\oplus \mathcal{H}_A^\perp\bigr).
$$
We may similarly 
regard
$S^2(\mathcal{H}_C,\mathcal{H}_B)$ and 
$S^2(\mathcal{H}_C,\mathcal{H}_A)$ as subspaces of $S^2(\mathcal{H})$.

The next statement means that for any 
$\phi\in L^\infty(\lambda_A\times\lambda_B\times\lambda_C)$,
$\Gamma^{A,B,C}(\phi)$ maps $S^2(\mathcal{H}_B,\mathcal{H}_A) 
\times S^2(\mathcal{H}_C,\mathcal{H}_B)$ into $S^2(\mathcal{H}_C,\mathcal{H}_A)$ and that 
under the previous identifications, this restriction `coincides' with $\Lambda(\phi)$.

\begin{proposition}\label{Connection}
Let $X\in S^2(L^2(\lambda_B), L^2(\lambda_A))$ 
and $Y \in S^2(L^2(\lambda_C), L^2(\lambda_B))$, and set
$$
\widetilde{X}=\rho_A \circ X \circ \rho_B^{-1} \in S^2(\mathcal{H}_B, \mathcal{H}_A)
\qquad\text{and}\qquad
\widetilde{Y}=\rho_B \circ Y \circ \rho_C^{-1} \in S^2(\mathcal{H}_C, \mathcal{H}_B).
$$
For any $\phi\in L^\infty(\lambda_A\times\lambda_B\times\lambda_C)$,
$\Gamma^{A,B,C}(\phi)(\widetilde{X},\widetilde{Y})$ 
belongs to $S^2(\mathcal{H}_C,\mathcal{H}_A)$ and 
\begin{equation}\label{subspace}
\Lambda(\phi)(X,Y) = \rho_A^{-1} \circ 
\Gamma^{A,B,C}(\phi)(\widetilde{X},\widetilde{Y}) \circ \rho_C.
\end{equation}
\end{proposition}

\begin{proof}
We first consider the special case
when $\phi = \chi_{F_1} \otimes \chi_{F_2} \otimes \chi_{F_3}$ 
for some measurable subsets $F_1 \subset \sigma(A), 
F_2 \subset \sigma(B)$ and $F_3 \subset \sigma(C)$.

Let $U \subset \sigma(A), V, V' \subset \sigma(B)$ and $W\subset 
\sigma(C)$ and consider the elementary 
tensors
$$
X = \chi_V \otimes \chi_U \in S^2(L^2(\lambda_B), L^2(\lambda_A))
\quad\hbox{and}\quad
Y=\chi_W \otimes \chi_{V'} \in S^2(L^2(\lambda_C), L^2(\lambda_B)).
$$ 
We associate
$\widetilde{X}$ and $\widetilde{Y}$ as in the statement.
Since $\rho_B\colon L^2(\lambda_B)\to \mathcal{H}_B$ is a unitary, we have  
$\rho_B^{-1}=\rho_B^{*}$ hence 
$$
\widetilde{X} = \rho_B(\chi_V)\otimes \rho_A(\chi_U) =E^B(V)e_2 \otimes E^A(U)e_1.
$$
Likewise,
$$
\widetilde{Y}=E^C(W)e_3 \otimes E^B(V')e_2.
$$

We have
\begin{align*}
\Lambda(\phi)(X,Y)
& = \int_{\sigma(B)} \phi(.,t_2,.)X(t_2,.)Y(.,t_2) \,\text{d}\lambda_B(t_2) \\
& = \int_{\sigma(B)} \chi_{F_2}(t_2)\chi_V(t_2)\chi_{V'}(t_2) 
\ \chi_{F_3}\chi_W \otimes \chi_{F_1}\chi_U \, \text{d}\lambda_B(t_2) \\
& = \left(\int_{F_2 \cap V \cap V'}\  \text{d}\lambda_B(t_2)\right)
\chi_{F_3 \cap W} \otimes \chi_{F_1 \cap U} \\
& = \lambda_B(F_2 \cap V \cap V')\, \chi_{F_3 \cap W} \otimes \chi_{F_1 \cap U}.
\end{align*}
Further using the above expressions of $\widetilde{X}$ and $\widetilde{Y}$, we have
\begin{align*}
\Gamma^{A,B,C}(\phi)(\widetilde{X},\widetilde{Y})
& = E^A(F_1)\widetilde{X}E^B(F_2)\widetilde{Y}E^C(F_3) \\
& = \bigl(E^B(V)e_2 \otimes E^A(F_1 \cap U)e_1\bigr)
\bigl(E^C(F_3\cap W)e_3 \otimes E^B(F_2 \cap V')e_2\bigr) \\
& = \left\langle E^B(F_2 \cap V')e_2,E^B(V)e_2\right\rangle
E^C(F_3\cap W)e_3 \otimes E^A(F_1 \cap U)e_1 \\
& = \left\langle E^B(F_2 \cap V'\cap V)e_2,e_2 \right\rangle 
E^C(F_3\cap W)e_3 \otimes E^A(F_1 \cap U)e_1 \\
& = \lambda_B(F_2 \cap V \cap V')\, E^C(F_3\cap W)e_3 \otimes E^A(F_1 \cap U)e_1.
\end{align*}
This shows that $\Gamma^{A,B,C}(\phi)(\widetilde{X},\widetilde{Y})$ 
belongs to $S^2(\mathcal{H}_C,\mathcal{H}_A)$ and that
(\ref{subspace}) holds true.

By linearity and continuity, this result holds as well
for all $X\in S^2(L^2(\lambda_B), L^2(\lambda_A))$ 
and all $Y \in S^2(L^2(\lambda_C), L^2(\lambda_B))$.

Finally since $\Lambda$ and $\Gamma^{A,B,C}$ are $w^*$-continuous, 
we deduce from the above special case
that the result actually holds true for all $\phi \in L^{\infty}(\lambda_A \times 
\lambda_B \times \lambda_C)$.
\end{proof}

\begin{corollary}\label{Iso}
The mapping $\Gamma^{A,B,C}$ from Theorem \ref{GammaABC} is an isometry.
\end{corollary}

\begin{proof}
Consider $\phi \in L^{\infty}(\lambda_A \times 
\lambda_B \times \lambda_C)$. For any
$X$ in $S^2(L^2(\lambda_B), L^2(\lambda_A))$ 
and any $Y$ in $S^2(L^2(\lambda_C), L^2(\lambda_B))$, we have  
\begin{align*}
\|\Lambda(\phi)(X,Y)\|_2
& = \|\rho_A^{-1} \circ \Gamma^{A,B,C}(\phi)(\widetilde{X},\widetilde{Y}) \circ \rho_C\|_2 \\
& \leq  \|\Gamma^{A,B,C}(\phi)(\widetilde{X},\widetilde{Y})\|_2 \\
& \leq \bignorm{\Gamma^{A,B,C}(\phi)}\| \widetilde{X} \|_2 \| \widetilde{X} \|_2
\end{align*}
by Proposition \ref{Connection}.
Since $\| \widetilde{X} \|_2=\norm{X}_2$ and $\| \widetilde{Y} \|_2=\norm{Y}_2$,
this implies that
\begin{equation}\label{comparison2}
\bignorm{\Lambda(\phi)}\leq \bignorm{\Gamma^{A,B,C}(\phi)}.
\end{equation}
By Proposition \ref{Function-case}, the left-hand side of this inequality is equal to
$\norm{\phi}_\infty$. Further $\Gamma^{A,B,C}$ is a contraction. Hence we obtain that 
$\norm{\Gamma^{A,B,C}(\phi)}=\norm{\phi}_\infty$.
\end{proof}

With a similar proof (left to the reader), one can show that the mapping 
$\Gamma^{A_1,\ldots, A_n}$ 
from Proposition \ref{nintegrals} in an isometry.

\section{$L^p_\sigma$-spaces}\label{Lp}

Let $(\Omega, \mu)$ be a $\sigma$-finite measure space and let 
$E$ be a Banach space. For any $1\leq p \leq +\infty$,
we let $L^p(\Omega;E)$ denote the classical Bochner space of measurable functions
$\varphi\colon\Omega\to E$ (defined up to almost everywhere zero functions)
such that the norm function $\norm{\varphi(\cdotp)}$
belongs to $L^p(\Omega)$ (see e.g. \cite[Chapter II]{Diestel}).

We will consider a dual version.
Assume that $E$ is separable. A function $\phi \colon \Omega \rightarrow E^*$ 
is said to be $w^*$-measurable if for all $x\in E$, the function
$t \in \Omega \mapsto \langle \phi(t), x \rangle$ is measurable. 
In this case, the function 
$t \in \Omega \mapsto \|\phi(t)\|$ is measurable. Indeed, if $(x_n)_n$ is a 
dense sequence in the unit sphere of $E$, 
then $\|\phi(.)\| = \sup_n \vert\langle \phi(.), x_n \rangle\vert$ is the 
supremum of a sequence of measurable functions, hence is measurable.

Let $1\leq q \leq +\infty$. By definition, $L^q_{\sigma}(\Omega;E^*)$ is 
the space of all $w^*$-measurable $\phi \colon \Omega \rightarrow E^*$ such that 
$\|\phi(.)\| \in L^q(\Omega)$, after taking quotient by the functions 
which are equal to $0$ almost
everywhere. We equip this space with
$$
\|\phi\|_q = \| \|\phi(.)\| \|_{L^q(\Omega)}.
$$
Then $(L^q_{\sigma}(\Omega; E^*), \|.\|_q)$ is a Banach space (the proof 
is the same as in the scalar case). Further by construction,
$L^q(\Omega;E^*)\subset L^q_{\sigma}(\Omega; E^*)$ isometrically.

Suppose that $1\leq p < +\infty$ and let $1 < q\leq +\infty$ 
be the conjugate exponent of $p$. For any 
$\phi \in L^q_{\sigma}(\Omega; E^*)$ and any $\varphi \in L^p(\Omega;E)$,
the function $t \mapsto \langle \phi(t), 
\varphi(t)\rangle$ is measurable. Indeed any element of $L^p(\Omega;E)$ 
is an almost everywhere
limit of a sequence of $L^p(\Omega)\otimes E$, hence it suffices
to check this fact 
when $\varphi \in L^p(\Omega)\otimes E$. In this case,
the measurablity of $\langle \phi(\cdotp), 
\varphi(\cdotp)\rangle$ is a straightforward consequence of the
$w^*$-measurability of $\phi$. 
By H\" older's inequality, the function $\langle \phi(\cdotp), 
\varphi(\cdotp)\rangle$ is actually integrable, which yields 
a duality pairing
\begin{equation}\label{pairing}
\langle \phi, \varphi \rangle := \int_{\Omega} 
\langle \phi(t), \varphi(t) 
\rangle \,\text{d}\mu(t)\,.
\end{equation}
Moreover we have
\begin{equation}\label{Holder}
\vert\langle \phi, \varphi \rangle\vert\leq\norm{\phi}_q\norm{\varphi}_p.
\end{equation}

\begin{theorem}\label{Lp-dual}
The duality pairing (\ref{pairing}) induces an isometric isomorphism
\begin{equation}\label{dualityLp}
L^p(\Omega;E)^* = L^q_{\sigma}(\Omega; E^*).
\end{equation}
\end{theorem}

The above theorem is well-known and has extensions to the non separable case.
However we have not found a satisfactory reference for this simple (=separable)
case and provide a proof below for the sake of
completeness. See \cite[Chapter IV]{Diestel} and the references therein for more
information.

Recall that we have $L^1(\Omega;E)^*=B(L^1(\Omega),E^*)$
by (\ref{L1tensorcor}). Hence in the case $p=1$, the above theorem
yields an isometric identification
\begin{equation}\label{DP}
L^\infty_\sigma(\Omega; F^*) = B(L^1(\Omega),F^*),
\end{equation}
a classical result going back to \cite[Theorem 2.1.6]{DunPet}.

\begin{proof}[Proof of Theorem \ref{Lp-dual}]
The inequality (\ref{Holder}) yields a contractive map
$\kappa\colon L^q_{\sigma}(\Omega; E^*)\to L^p(\Omega;E)^*$. Our aim is to
show that $\kappa$ is an isometric isomorphism.

According to the separability assumption there exists a nondecreasing
sequence $(E_n)_{n\geq 1}$ of finite dimensional subspaces of $E$ such that 
$\cup_n E_n$ is dense in $E$. Since $E_n$ is finite dimensional, 
$L^q_{\sigma}(\Omega, E_n^*)= L^q(\Omega, E_n^*)$ and $E_n$ satisfies the 
conclusion of the theorem to be proved, that is,
\begin{equation}\label{Dual-En}
L^p(\Omega;E_n)^* = L^q(\Omega; E_n ^*)
\end{equation}
isometrically (see \cite[Chapter IV]{Diestel}). In the sequel 
we regard $L^p(\Omega;E_n)$ 
as a subspace of $L^p(\Omega;E)$ in a natural way.

We first note that $\kappa$ is 1-1. Indeed if $\phi\in L^q_\sigma(\Omega;E^*)$
is such that $\kappa(\phi)=0$, then for any $n\geq 1$,
$\phi(t)_{\vert E_n} = 0$ a.e. by (\ref{Dual-En}). Hence 
$\phi(t)_{\vert \cup_n E_n} = 0$ a.e., 
which implies that $\phi(t) = 0$ a.e.

Now let $\delta\in L^p(\Omega;E)^*$, with $\norm{\delta}\leq 1$. Applying 
(\ref{Dual-En}) to the restriction of $\delta$ to $L^p(\Omega;E_n)$
we obtain, for any $n\geq 1$, a measurable function
$\phi_n\colon \Omega\to E_n^*$ such that 
$\norm{\phi_n}_q\leq 1$ and
$$
\forall \,
\varphi\in L^p(\Omega)\otimes E_n,\qquad
\delta(\varphi) = \int_{\Omega} 
\langle \phi_n(t), \varphi(t) 
\rangle \,\text{d}\mu(t)\,.
$$
We may assume that for any $n\geq 1$, we have 
\begin{equation}\label{Restriction}
\forall\, t\in \Omega,\qquad  {\phi_{n+1}(t)}_{\vert E_n}= 
\phi_{n}(t).
\end{equation}
Indeed by construction, ${\phi_{n+1}}_{\vert E_n}= 
\phi_{n}$ a.e. and 
the family $(\phi_n)_{n\geq 1}$ is countable so we can
modify all the functions $\phi_n$ on a common negligible set to
get (\ref{Restriction}).

It follows  that for any $t\in\Omega$,
$(\norm{\phi_n(t)})_{n\geq 1}$
is a nondecreasing sequence, so we can define a measurable
$\nu\colon\Omega\to [0,\infty]$ by
$$
\nu(t) = \lim_n\norm{\phi_n(t)},\qquad t\in \Omega.
$$
If $q<\infty$ we may write
$$
\int_\Omega\nu(t)^q\,\text{d}\mu(t)\, 
= \lim_n\int_\Omega\norm{\phi_n(t)}^q\,\text{d}\mu(t)\,\leq 1,
$$
by the monotone convergence theorem. This implies that
$\nu$ is a.e. finite. If $q=\infty$, the fact that 
$\norm{\phi_n}_\infty\leq 1$ for any $n\geq 1$ implies that
$\nu(t)\leq 1$ for a.e. $t\in\Omega$.
Thus in any case, there exists a negligible
subset $\Omega_0\subset\Omega$ such that $\nu(t)<\infty$ for any
$t\in\Omega\setminus\Omega_0$.

If $t\in\Omega\setminus\Omega_0$, then by (\ref{Restriction}) and the density
of $\cup_n E_n$, there exists a unique element of $E^*$, that we call
$\phi(t)$, such that
$$
\forall\,n\geq 1, \ \forall\, x\in E_n,
\qquad \langle \phi(t),x\rangle 
= \langle \phi_n(t),x\rangle.
$$
Next we set $\phi(t)=0$ for any $t\in\Omega_0$.
We thus have a function $\phi\colon\Omega\to E^*$.

Let $x\in E$ and let $(x_j)_j$ be  a sequence
of $\cup_n E_n$ converging to $x$. Then
$\langle \phi(\cdotp),x_j\rangle\to\langle \phi(\cdotp),x\rangle$ pointwise.
Moreover for any $j$, the function
$\langle \phi(\cdotp),x_j\rangle$ is measurable by construction, hence 
$\langle \phi(\cdotp),x\rangle$ is measurable. Thus $\phi$
is $w^*$-measurable.

Now from the definition of 
$\phi$, we see that
$\delta$ and $\kappa(\phi)$ coincide on $L^p(\Omega)\otimes E_n$
for any $n\geq 1$. Consequently, $\delta=\kappa(\phi)$.
Moreover $\norm{\phi}_q = 
\lim_n\norm{\phi_n}_q\leq 1$. 

This proves that $\kappa$ is a metric surjection,
and hence an isometric isomorphism.
\end{proof}

\begin{remark}\label{RNP}
We already noticed that $L^q_\sigma(\Omega; E^*)=L^q(\Omega; E^*)$ when $E$ is finite dimensional.
It turns out that for a general Banach space $E$, the equality 
$L^q_\sigma(\Omega; E^*)=L^q(\Omega; E^*)$ is equivalent to $E^*$ having the
Radon-Nikodym property, see e.g. \cite[Chapter IV]{Diestel}. All Hilbert spaces
(more generally all reflexive Banach spaces) have the Radon-Nikodym property. Later on 
we will use this property that for any separable Hilbert space $H$ and 
any $1\leq q\leq\infty$, we have
$$
L^q_\sigma(\Omega;H) = L^q(\Omega;H).
$$
\end{remark}

Let $E$ and $F$ be two separable Banach spaces. Being a completion of
$E\otimes F$, 
their projective tensor product
$E \overset{\wedge}{\otimes} F$ is separable as well. Recall that its dual space is equal
to $B(E, F^*)$. Whenever $\phi\colon\Omega\to B(E, F^*)$ is
a $w^*$-measurable function, then for any $x\in E$, the function
$T_\phi(x)\colon \Omega\to F^*$ defined by
\begin{equation}\label{Actiondual}
\bigl[T_\phi(x)\bigr](t) = \bigl[\phi(t)\bigr](x),
\qquad t\in\Omega,
\end{equation}
is $w^*$-measurable.

\begin{corollary}\label{CoroDuality} 
The mapping $\phi\mapsto T_\phi$ given by (\ref{Actiondual})
induces an isometric isomorphism
$$
B(E, L^{\infty}_{\sigma}(\Omega, F^*))=
L^{\infty}_{\sigma}(\Omega; B(E, F^*)).
$$
\end{corollary}

\begin{proof}
By Theorem \ref{Lp-dual} for $p=1$,
and by (\ref{dualproj}) and (\ref{L1tensor}), 
we have  isometric isomorphisms
\begin{align*}
B(E, L^{\infty}_{\sigma}(\Omega; F^*))
& = \bigl(E \overset{\wedge}{\otimes} L^1(\Omega;F) \bigr)^*\\
& = \bigl( E \overset{\wedge}{\otimes}  L^1(\Omega) \overset{\wedge}{\otimes}  F \bigr)^* \\
& = L^1(\Omega; E \overset{\wedge}{\otimes}  F)^* \\
& = L^{\infty}_{\sigma}(\Omega; B(E, F^*)).
\end{align*}
It is easy to check that the correspondence is given by $(\ref{Actiondual})$.
\end{proof}

\begin{remark}\label{WeakTensorisation}
Let $E_1, E_2$ be two Banach spaces and let $U\colon E_1^*\to E_2^*$
be a $w^*$-continuous map. For any $\phi\in L^\infty_\sigma(\Omega; E_1^*)$,
the composition map $U\circ\phi\colon \Omega\to E_2^*$ belongs to 
$L^\infty_\sigma(\Omega; E_2^*)$ and the mapping $\phi\mapsto U\circ \phi$
is a bounded operator from $L^\infty_\sigma(\Omega; E_1^*)$
into $L^\infty_\sigma(\Omega; E_2^*)$, whose norm is equal $\norm{U}$. 
It is easy to check that this mapping is $w^*$-continuous.
If further $U$ is an isometry, 
then $\phi\mapsto U\circ \phi$ is an isometry as well.

Applying this elementary principle to the embedding of 
$\Gamma_2(L^1(\Omega_1),L^\infty(\Omega_2))$ into the space
$$
B\bigl(\mathcal{K}(L^2(\Omega_1), L^2(\Omega_2)), B(L^2(\Omega_1), L^2(\Omega_2))\bigr),
$$
provided by Theorem \ref{Pellerthm}, we obtain a $w^*$-continuous isometric inclusion
\begin{equation}\label{L-infty-inclusion}
L^\infty_\sigma\bigl(\Omega; \Gamma_2(L^1(\Omega_1),L^\infty(\Omega_2))\bigr)
\,\subset\, 
L^\infty_\sigma\bigl(\Omega; 
B(\mathcal{K}(L^2(\Omega_1), L^2(\Omega_2)), B(L^2(\Omega_1), L^2(\Omega_2))
)\bigr).
\end{equation}
\end{remark}

\vskip 1cm

\section{Measurable factorization in 
$L^{\infty}_{\sigma}(\Omega; \Gamma_2(E, F^*))$}\label{factorization}

The main purpose of this section is to prove Theorem \ref{theofacto} below. 
This result will be applied
in Subsection \ref{Special} (and in Section \ref{mainresult}) 
to the study of measurable Schur multipliers.

We will say that a  measure space $(\Omega,\mu)$ is separable
when $L^2(\Omega,\mu)$ is separable. This implies that $(\Omega,\mu)$
is $\sigma$-finite and moreover,
$L^p(\Omega,\mu)$ is separable for any $1\leq p<\infty$.

\subsection{The general case}\label{General}

It follows from Subsection \ref{tensorproducts} that for
any separable Banach spaces $E,F$, the space
$\Gamma_2(E,F^*)$ is a dual space with a separable predual.
If $H$ is a separable Hilbert space, then 
$B(E,H)$ and $B(F,H)$ are also dual spaces with separable predual.

\begin{theorem}\label{theofacto}
Let $(\Omega,\mu)$ be a separable measure space and 
let $E,F$ be two separable Banach spaces.
Let $\phi\in L^\infty_\sigma\bigl(\Omega;\Gamma_2(E,F^*)\bigr)$. Then there 
exist a separable Hilbert space $H$ and two functions
$$
\alpha\in L^\infty_\sigma\bigl(\Omega;B(E,H)\bigr)
\qquad\hbox{and}\qquad
\beta\in L^\infty_\sigma\bigl(\Omega;B(F,H)\bigr)
$$
such that $\norm{\alpha}_\infty\norm{\beta}_\infty\leq\norm{\phi}_\infty$ and
for any $(x,y)\in E\times F$,
\begin{equation}\label{factor}
\bigl\langle [\phi(t)](x),y\bigr\rangle = \bigl\langle [\alpha(t)](x), [\beta(t)](y)
\bigr\rangle,\qquad \hbox{for a.e.}\ t\in\Omega.
\end{equation}
\end{theorem}

We will need two lemmas, in which $(\Omega,\mu)$ denotes an arbitrary 
$\sigma$-finite measure space.

The first one is a variant of the classical classification of 
abelian von Neumann algebras.
For any $\theta\in L^\infty(\Omega)$, and any Hilbert space $H$, 
we let $M_\theta\colon L^2(\Omega;H)\to L^2(\Omega;H)$ 
denote the multiplication operator taking 
any $\varphi\in L^2(\Omega;H)$ to $\theta\varphi$.

\begin{lemma}\label{multip} 
Let $\H$ be a separable Hilbert space and let 
$\pi\colon L^\infty(\Omega)\to B(\H)$
be a $w^*$-continuous $*$-representation. 
There exist a separable Hilbert space $H$ and an
isometric embedding 
$\rho\colon \H\hookrightarrow L^2(\Omega;H)\,$ 
such that for any $\theta\in L^\infty(\Omega)$,
$$
\rho \pi(\theta) = M_\theta\rho.
$$
\end{lemma}

\begin{proof}
Since $\pi$ is $w^*$-continuous, there exists a measurable subset $\Omega'\subset\Omega$ such that
the range of $\pi$ is isomorphic to $L^\infty(\Omega')$ in the von Neumann
algebra sense and $\pi$ coincides with the restriction map (apply \cite[Corollary 2.5.5]{Ped}).
It therefore follows from \cite[Theorem II.3.5]{Davidson} that there exist 
a measurable partition $\{\Omega_n\, :\, 1\leq n\leq \infty\}$ 
of $\Omega'$ and a unitary operator
$$
\rho_1\colon \H\longrightarrow \oplus_{1\leq n\leq\infty}^2 L^2(\Omega_n;\ell^2_n) 
$$
such that for any $\theta\in L^\infty(\Omega)$, $\rho_1\pi(\theta)\rho_1^*$ 
coincides with the multiplication
by $\theta$. (Note that in the above decomposition, 
the index $n$ may be finite or infinite
and the notation $\ell^2_\infty$ stands for $\ell^2$.)
Let
$$
H= \overset{2}{\oplus}_{1\leq n\leq\infty}  \ell^2_n
$$
and consider the canonical embedding 
$$
\rho_2\colon \oplus_{1\leq n\leq\infty}^2 L^2(\Omega_n;\ell^2_n)\longrightarrow L^2(\Omega;H).
$$
Then $\rho=\rho_2\rho_1$ satisfies the lemma. 
\end{proof}

It is well-known that for any Hilbert space $H$, the commutant of
$$
L^\infty(\Omega)\simeq L^\infty(\Omega)\otimes I_H
\,\subset\,B(L^2(\Omega;H))
$$
coincides with $L^\infty(\Omega)\overline{\otimes} B(H)$. 
The next statement is a generalization of this result to the case when $H$ is 
replaced by Banach spaces.

We consider two separable Banach spaces $W_1,W_2$.
Note that by (\ref{dualproj}), $B(W_1,W_2^*)$ is a dual space with separable predual.
We say that a linear map
$$
T\colon L^2(\Omega;W_1)\longrightarrow L^2_\sigma(\Omega; W_2^*)
$$ 
is a module map provided that
$$
\forall\,\varphi\in L^2(\Omega;W_1),\ \forall\, \theta\in L^\infty(\Omega),\qquad
T(\theta \varphi) = \theta T(\varphi).
$$
Next we generalize the notion of multiplication by 
an $L^\infty$-function as follows.
For any $\Delta\in L^\infty_\sigma\bigl(\Omega;B(W_1,W_2^*)\bigr)$, we  
define a multiplication operator
\begin{equation}\label{Multiplication}
M_\Delta\colon L^2(\Omega;W_1)\longrightarrow L^2_\sigma(\Omega;W_2^*)
\end{equation}
by setting 
$$
\bigl[M_\Delta(\varphi)\bigr](t)= [\Delta(t)](\varphi(t)),\qquad t\in\Omega,
$$
for any $\varphi\in L^2(\Omega;W_1)$. 
Indeed it is easy to check (left to the reader) that
the function in the right-hand side 
of the above equality belongs to $L^2_\sigma(\Omega;W_2^*)$. Moreover 
\begin{equation}\label{Multiplication2}
\norm{M_\Delta}
=\norm{\Delta}_\infty.
\end{equation}
Each multiplication operator $M_\Delta$ is a module map, as we have
$$
M_\Delta(\theta\varphi)=M_{\Delta\theta}(\varphi) = \theta M_\Delta(\varphi)
$$
for any $\theta\in L^\infty(\Omega)$.
The following lemma is a converse.

\begin{lemma}\label{bimod} 
Let $T\colon L^2(\Omega;W_1)
\to L^2_\sigma(\Omega;W_2^*)$ be a  module map. 
Then there exists a function $\Delta\in 
L^\infty_\sigma\bigl(\Omega;  B(W_1,W_2^*)\bigr)$ such that $T=M_\Delta$.
\end{lemma}

\begin{proof} In the scalar case $(W_1=W_2=\Cdb)$ 
this is an elementary result; the proof
consists in reducing to this scalar case.

We define a bilinear map
$\widehat{T}\colon W_1\times W_2\to B(L^2(\Omega))$ by the following formula.
For any $w_1\in W_1$, $w_2\in W_2$ and $x\in L^2(\Omega)$, we set
$$
\bigl[\widehat{T}(w_1,w_2)\bigr](x) = 
\bigl\{t \mapsto \bigl\langle 
\bigl[T(x\otimes w_1)\bigr](t),w_2\bigr\rangle\bigr\}.
$$
Recall the identification
$L^2_\sigma(\Omega;W_2^*)=L^2(\Omega;W_2)^*$ from Theorem \ref{Lp-dual}. 
If we consider $T$ as a map from $L^2(\Omega;W_1)$ into $L^2(\Omega;W_2)^*$, then 
we have
\begin{equation}\label{T}
\bigl\langle T(x\otimes w_1), y\otimes w_2\bigr\rangle
= \int_\Omega \Bigl(\bigl[\widehat{T}(w_1,w_2)
\bigr](x)\Bigr)(t)\, y(t) \,\text{d}\mu(t)
\end{equation}
for any $w_1\in W_1$, $w_2\in W_2$, $x\in L^2(\Omega)$ and $y\in L^2(\Omega)$.

Further for any $\theta\in L^\infty(\Omega)$ and $x\in L^2(\Omega)$, we have
\begin{align*}
\bigl[\widehat{T}(w_1,w_2)\bigr](\theta x) &= 
\bigl\langle 
\bigl[T(\theta(x\otimes w_1))\bigr](\cdotp), w_2\bigr\rangle\\
& = \bigl\langle 
\theta(\cdotp)\,\bigl[T(x\otimes w_1)\bigr](\cdotp), w_2\bigr\rangle\\
& = \theta\bigl[\widehat{T}(w_1,w_2)\bigr](x),
\end{align*}
because $T$ is a module map. Hence $\widehat{T}(w_1,w_2)$ is a module map.

Let us identify $L^\infty(\Omega)$ with the von Neumann
subalgebra of $B(L^2(\Omega))$ consisting of multiplication
operators. The above property shows that 
$\widehat{T}(w_1,w_2)$ is such a multiplication operator for
any  
$w_1\in Z_1$ and $w_2\in Z_2$. Hence we may 
actually regard
$\widehat{T}$ as a bilinear map
$$
\widehat{T}\colon W_1\times W_2\longrightarrow L^\infty(\Omega).
$$
Now observe that applying (\ref{biliproj}), (\ref{dualproj}) and (\ref{DP}), 
we have isometric identifications
\begin{align*}
B_2(W_1\times W_2, L^\infty(\Omega)) & = 
B(W_1\overset{\wedge}{\otimes} W_2,L^\infty(\Omega)) \\
& = B(L^1(\Omega), (W_1\overset{\wedge}{\otimes} W_2)^*) \\
& = B(L^1(\Omega), B(W_1,W_2^*)) \\
& = 
L^\infty_\sigma\bigl(\Omega; B(W_1,W_2^*)\bigr).
\end{align*}
Let $\Delta\in L^\infty_\sigma\bigl(\Omega; B(W_1,W_2^*)\bigr)$ be 
corresponding to 
$\widehat{T}$ in this identification. Then we have 
$$
\bigl\langle [\Delta(t)](w_1),w_2\bigr\rangle
= \bigl(\widehat{T}(w_1,w_2)\bigr)(t),\qquad
w_1\in W_1,\,w_2\in W_2,\, t\in \Omega.
$$
Thus applying (\ref{T}) we obtain that 
\begin{align*}
\bigl\langle T(x\otimes w_1), y\otimes w_2\bigr\rangle
& =\int_{\Omega} \bigl\langle [\Delta(t)](w_1),w_2\bigr\rangle\, x(t)y(t)\, \text{d}\mu(t)\\
& = \bigl\langle M_\Delta (x\otimes w_1), y\otimes w_2\bigr\rangle
\end{align*}
for any $w_1\in W_1$, $w_2\in W_2$, $x\in L^2(\Omega)$ and $y\in L^2(\Omega)$.
By the density of $L^2(\Omega)\otimes W_1$ and $L^2(\Omega)\otimes W_2$
in $L^2(\Omega;W_1)$ and $L^2(\Omega;W_2)$, respectively, this implies that $T=M_\Delta$.
\end{proof}

\begin{proof}[Proof of Theorem \ref{theofacto}.]
This proof should be regarded as a module version of the proof of \cite[Theorem 3.4]{PisierBook}.
As in this book we adopt the following notation. For any finite families $(f_j)_j$
and $(e_i)_i$ in $E$, we write
$$
(f_j)_j < (e_i)_i 
$$
provided that
$$
\forall\,\eta\in E^*,\qquad \sum_j\vert
\eta(f_j)\vert^2\leq \sum_i\vert
\eta(e_i)\vert^2.
$$

In the sequel we simply write $L^2$ (resp. $L^\infty$) instead of $L^2(\Omega)$ 
(resp. $L^\infty(\Omega)$) as there is no risk of confusion. Then we set
$$
V=L^2 \otimes E\subset L^2(\Omega;E).
$$

We fix some  $\phi\in L^\infty_\sigma\bigl(\Omega;\Gamma_2(E,F^*)\bigr)$ and we
let $C=\norm{\phi}_\infty$. Then $\phi$ is an element of
$L^\infty_\sigma\bigl(\Omega;B(E,F^*)\bigr)$. Hence according to (\ref{Multiplication})
we may consider the multiplication operator
$$
T = M_\phi\colon  L^2(\Omega;E)\longrightarrow L^2_\sigma(\Omega; F^*).
$$

We let $I= L^\infty\times E^*$. A generic element of $I$ will be denoted by 
$\zeta=(\theta,\eta)$, with $\theta\in L^\infty$ and $\eta\in E^*$.

For any $v=\sum_s x_s\otimes e_s\in V\,$ (finite sum)  
and $\zeta=(\theta,\eta)\in I$, we set
$$
\zeta\cdotp v = \sum_s \eta(e_s) \theta x_s \,\in\, L^2.
$$

\begin{lemma}\label{ineq2} 
Let $(w_j)_j$ and $(v_i)_i$ be finite families in $V$ such that
\begin{equation}\label{dom}
\forall\,\zeta\in I,\qquad
\sum_j\norm{\zeta\cdotp w_j}^2_2\leq 
\sum_i\norm{\zeta\cdotp v_i}^2_2.
\end{equation}
Then
\begin{equation}\label{dom1}
\sum_j\norm{T(w_j)}^2_2\leq C^2
\sum_i\norm{v_i}^2_2.
\end{equation}
\end{lemma}

\begin{proof}
Let $(w_j)_j$ and $(v_i)_i$ be finite families in $V$ and assume (\ref{dom}). 
Consider $e_{i,s}, f_{j,s}$ in $E$,
$x_{i,s}, y_{j,s}$ in $L^2$ such that 
$$
v_i = \sum_s x_{i,s}\otimes e_{i,s}
\qquad\hbox{and}\qquad
w_j = \sum_s y_{j,s}\otimes f_{j,s}.
$$
Let $\zeta=(\theta,\eta)\in I$. For any $j$,
$$
\norm{\zeta\cdotp w_j}^2_2\,=\,\int_{\Omega}\Bigl\vert
\sum_s\eta(f_{j,s})\theta(t)y_{j,s}(t)\Bigr\vert^2\, \text{d}\mu(t)\,.
$$
Hence
$$
\sum_j\norm{\zeta\cdotp w_j}^2_2\,=\, \int_{\Omega}\vert \theta(t)\vert^2
\Bigl(\sum_j\Bigl\vert\sum_s \eta(f_{j,s})y_{j,s}(t)\Bigr\vert^2\Bigr)\, \text{d}\mu(t)\,.
$$
Likewise,
$$
\sum_i\norm{\zeta\cdotp v_i}^2_2\,=\, \int_{\Omega}\vert \theta(t)\vert^2
\Bigl(\sum_i\Bigl\vert\sum_s \eta(e_{i,s})x_{i,s}(t)\Bigr\vert^2\Bigr)\, \text{d}\mu(t)\,.
$$
Thus by (\ref{dom}), we have
\begin{equation}\label{sep}
\int_{\Omega}\vert \theta(t)\vert^2
\Bigl(\sum_j\bigl\vert\eta\bigl(w_j(t)\bigr)\bigr\vert^2\Bigr)\,
\text{d}\mu(t)\,\leq
\int_{\Omega}\vert \theta(t)\vert^2\Bigl(\sum_i\bigl\vert\eta
\bigl(v_i(t)\bigr)\bigr\vert^2\Bigr)\,
\text{d}\mu(t)\,.
\end{equation}

Let $E_1\subset E$ be the subspace spanned by the $e_{i,s}$ and $f_{j,s}$. 
Since it is finite dimensional,
its dual space is obviously separable. Let
$(\eta_n)_{n\geq 1}$ be a dense sequence of $E_1^*$ and
for any $n\geq 1$, extend $\eta_n$ to an element of $E^*$
(still denoted by $\eta_n$). Then for any
finite families $(f_j)_j$ and $(e_i)_i$ in $E_1$, we have 
$$
(y_j)_j<(x_i)_i\,\Longleftrightarrow\,\forall\, n\geq 1,\quad 
\sum_j\vert\eta_n(f_j)\vert^2\leq \sum_i\vert\eta_n(e_i)\vert^2.
$$
It follows from (\ref{sep}) that for almost every $t\in \Omega$, we have
$$
\sum_j\vert\eta_n\bigl(w_j(t)\bigr)\vert^2\leq \sum_i\vert\eta_n\bigl(v_i(t)\bigr)\vert^2
$$
for every $n\geq 1$. Since the functions $v_i,w_j$ are valued in $E_1$, this implies that
$$
(w_j(t))_j < (v_i(t))_i\qquad\hbox{for a.e.}\ t\in \Omega.
$$
By the implication `(i) $\Rightarrow$ (iii)' of  \cite[Theorem 3.4]{PisierBook}, 
this property 
implies that for a.e. $t\in\Omega$,
$$
\sum_j\bignorm{[\phi(t)]\bigl(w_j(t)\bigr)}^2_{F^*}\leq 
C^2 \sum_i\bignorm{v_i(t)}^2_E.
$$
Integrating this inequality on $\Omega$ yields $(\ref{dom1})$.
\end{proof}

We let $\Lambda$ be the set of all functions $g\colon I\to \Rdb$ for which there
exists a finite family $(v_i)_i$ in $V$ such that 
\begin{equation}\label{Lambda}
\forall\,\zeta\in I,\qquad
\vert g(\zeta)\vert\leq \sum_i\norm{\zeta\cdotp v_i}^2_2.
\end{equation}
This is a real vector space. We let $\Lambda_+$ denote  
its positive part, i.e. the set of all functions $I\to\Rdb_+\,$ belonging to 
$\Lambda$. This is a convex cone.
For any $g\in\Lambda$ we set
$$
p(g)=C^2\inf\Bigl\{\sum_i\norm{ v_i}^2_2\Bigr\},
$$
where the infimum runs over all finite families $(v_i)_i$ in $V$ satisfying (\ref{Lambda}).
It is easy to check that $p$ is sublinear, that is, 
$p(g + g')\leq p(g) + p(g')$ for any $g,g'\in \Lambda$ and 
$p(tg)=tp(g)$ for any $g\in\Lambda$ and any $t\geq 0$.

Next for any $g\in\Lambda_+$, we set
$$
q(g)=\sup\Bigl\{\sum_j\norm{T(w_j)}^2_2\Bigr\},
$$
where the supremum runs over all finite families $(w_j)_j$ in $V$ satisfying 
\begin{equation}\label{Lambda+}
\forall\zeta\in I,\qquad g(\zeta)\geq \sum_j \norm{\zeta\cdotp w_j}^2_2.
\end{equation}
It is easy to check that $q$ is superlinear, that is,
$q(g) + q(g')\leq q(g + g')$ for any $g,g'\in \Lambda_+$ and 
$q(tg)=tq(g)$ for any $g\in\Lambda_+$ and any $t\geq 0$.

By Lemma $\ref{ineq2}$,
$q\leq p$ on $\Lambda_+$. Hence by the Hahn-Banach Theorem 
given in \cite[Corollary 3.2]{PisierBook}, 
there exists a positive linear functional $\ell\colon\Lambda\to\Rdb\,$ such that
\begin{equation}\label{sep1}
\forall\, g\in\Lambda,\qquad \ell(g)\leq p(g)
\end{equation}
and
\begin{equation}\label{sep2}
\forall\, g\in\Lambda_+,\qquad q(g)\leq \ell(g).
\end{equation}

Following \cite[Chapter 8]{PisierBook}, we introduce a Hilbert space
$$
\Lambda_2(I,\ell; L^2)
$$
defined as follows. First we let $\Ll(I,\ell;L^2)$ be the 
set of all functions $G\colon I\to L^2$
such that the $\Rdb$-valued function 
$\zeta\mapsto\norm{G(\zeta)}^2_2$ belongs to $\Lambda$
and we set $N(G)=\bigl(\ell(\zeta\mapsto\norm{G(\zeta)}^2_2)
\bigr)^{\frac12}$
for any such function. Then $\Ll(I,\ell;L^2)$ is a complex vector space and
$N$ is a Hilbertian seminorm on $\Ll(I,\ell;L^2)$. Hence the quotient of
$\Ll(I,\ell;L^2)$ by the kernel of $N$ is a pre-Hilbert space.
By definition, $\Lambda_2(I,\ell; L^2)$ is the completion of this
quotient space.

For any $v\in V$, the function $\zeta\mapsto\zeta\cdotp v$ belongs to 
$\Ll(I,\ell;L^2)$.
Then we define a linear map
$$
T_1\colon V\longrightarrow \Lambda_2(I,\ell; L^2)
$$
as follows: for any $v\in V$, $T_1(v)$ is the class of $\zeta\mapsto\zeta\cdotp v$
modulo the kernel of $N$. 
Then we have
\begin{align*}
\norm{T_1(v)}^2_L &  = \ell\bigl(\zeta\mapsto \norm{\zeta\cdotp v}^2\bigr)\\
& \leq p\bigl(\zeta\mapsto \norm{\zeta\cdotp v}^2_2\bigr)\\
& \leq C^2\norm{v}^2_2
\end{align*}
by $(\ref{sep1})$ and the definition of $p$.
Hence $T_1$ uniquely extends to a bounded operator 
$$
T_1\colon L^2(\Omega; E)\longrightarrow \Lambda_2(I,\ell; L^2),\qquad\hbox{with}\ \norm{T_1}\leq C.
$$

For any $v\in V$, we have
$$
\norm{T(v)}^2_2\leq q\bigl(\zeta\mapsto \norm{\zeta\cdotp v}^2\bigr)\leq 
\ell\bigl(\zeta\mapsto \norm{\zeta\cdotp v}^2\bigr) = \norm{T_1(v)}^2.
$$
The resulting inequality $\norm{T(v)}_2\leq \norm{T_1(v)}$
implies the existence of a (necessarily unique) bounded linear operator
$$
T_2\colon \overline{T_1(V)}\longrightarrow L^2_\sigma(\Omega;F^*),
\qquad\hbox{with}\ \norm{T_2}\leq 1,
$$
such that 
\begin{equation}\label{D}
\forall\, v\in V,\qquad T(v) = T_2\bigl(T_1(v)\bigr).
\end{equation}
(Here and later on in the paper, $\overline{T_1(V)}\subset 
\Lambda_2(I,\ell; L^2)$ denotes the closure of $T_1(V)$.)

For any $v\in V$ and any $\theta\in L^\infty$, we have
\begin{equation}\label{T1}
\norm{T_1(\theta v)}\leq \norm{\theta}_\infty\norm{T_1(v)}.
\end{equation}
Indeed write $v=\sum_s x_s\otimes e_s\,$, with $e_s\in E$ 
and $x_s\in L^2$. For any $\gamma\in L^\infty$ and 
$\eta\in E^*$, we have
$$
\Bignorm{\sum_s \eta(e_s)\gamma \theta x_s}_{2}\leq \norm{\theta}_\infty
\Bignorm{\sum_s \eta(e_s) \gamma x_s}_{2}.
$$
Hence $\norm{\zeta\cdotp(\theta v)}\leq \norm{\theta}_\infty 
\norm{\zeta\cdotp v}$ for any $\zeta=(\gamma,\eta)
\in I$. Since the functional 
$\ell$ is positive on $\Lambda$, this implies that
$\ell\bigl(\zeta\mapsto \norm{\zeta\cdotp (\theta v)}^2\bigr)
\leq \norm{\theta}_\infty^2 \ell\bigl(
\zeta\mapsto \norm{\zeta\cdotp v}^2\bigr)$,
which yields (\ref{T1}).

This inequality implies the existence of 
a (necessarily unique) linear contraction
$$
\pi\colon L^\infty\longrightarrow \B\bigl(\overline{T_1(V)}\bigr),
$$
such that 
\begin{equation}\label{commute}
T_1(\theta v) =\pi(\theta)T_1(v),
\qquad v\in L^2(\Omega;E),\, \theta\in L^\infty.
\end{equation}
It is clear that $\pi$ is a unital homomorphism. This implies that
$\pi$ is a  $*$-representation. Indeed for any unitary $\theta\in L^\infty$,
we have $I = \pi(\overline{\theta}\theta)=
\pi(\overline{\theta})\pi(\theta)=\pi(\theta)\pi(\overline{\theta})$
and the two operators $\pi(\overline{\theta})$ and $\pi(\theta)$ are contractions.
This implies that $\pi(\theta)$ is a unitary and that 
\begin{equation}\label{*}
\pi(\theta)^*=\pi(\overline{\theta})
\end{equation}
Since unitaries generate $L^\infty$, (\ref{*}) actually holds true for any 
$\theta\in L^\infty$.

Let $\theta\in L^\infty$ and assume that $(\theta_\iota)_\iota$ 
is a bounded net of $L^\infty$
converging to $\theta$ in the $w^*$-topology. 
For any $x\in L^2$, $\theta_\iota x\to \theta x$ in $L^2$
(this uses the boundedness of the net).
By the continuity of $T_1$ this implies that 
for any $e\in E$, $T_1(\theta_\iota x\otimes e)\to
T_1(\theta x\otimes e)$ in $\overline{T_1(V)}$. By linearity, this implies
that for any $v\in V$, $T_1(\theta_\iota v)\to
T_1(\theta v)$ in $\overline{T_1(V)}$. In other words,
$\pi(\theta_\iota)(h)\to \pi(\theta)(h)$ for any $h\in T_1(V)$. 
Since the net $(\pi(\theta_\iota))_\iota$ is bounded, this implies that 
$\pi(\theta_\iota)\to \pi(\theta)$ strongly. Hence $\pi$ is a $w^*$-continuous
$*$-representation.

Recall that $E$ and $L^2$ are assumed separable, hence 
the Hilbert space $\overline{T_1(V)}$ is separable. 
By Lemma $\ref{multip}$, there exists a 
separable Hilbert space $H$ and an isometric embedding
$\rho\colon \overline{T_1(V)}\hookrightarrow L^2(\Omega;H)$ such that 
$\rho\pi(\theta)=M_\theta\rho\,$ for any $\theta\in L^\infty$. 
Then for any such $\theta$ and any $v\in L^2(\Omega;E)$, we have
$$
\rho T_1 (\theta v) = \bigl[\rho \pi(\theta) T_1\bigr](v)  
= \theta \rho(T_1(v)),
$$
by (\ref{commute}). This shows that the composed map
$$
S_1 = \rho T_1\colon L^2(\Omega;E)\longrightarrow L^2(\Omega;H)\quad\hbox{is a module map}.
$$

Define 
$$
S_2 = T_2\rho^* \colon L^2(\Omega;H) \longrightarrow  L^2_\sigma(\Omega;F^*).
$$
Let $\theta\in L^\infty(\Omega)$. For any $v\in V$, we have
$$
\bigl[T_2\pi(\theta)\bigr](T_1(v)) = T_2 T_1 (\theta v) =T(\theta v) =\theta T(v)
=\theta T_2(T_1(v))
$$
by (\ref{commute}),  (\ref{D}) and the fact that $T$ is a module map. This shows that
$$
T_2\pi(\theta) =  M_\theta T_2.
$$
Further we have 
$\rho^*M_\theta = \bigl(M_{\overline{\theta}}\rho\bigr)^* = \bigl(\rho\pi(\overline{\theta})\bigr)^*=
\pi(\theta)\rho^*$. Hence $M_\theta S_2 = S_2 M_\theta$, that is, 
$$
S_2 \quad\hbox{is a module map}.
$$

Since $\rho^*\rho$ is equal to the identity of $\overline{T_1(V)}$, 
it follows from (\ref{D}) that 
$$
T= S_2 S_1.
$$
Thus we have constructed a `module Hilbert space factorization' of $T$, 
and this is the main point.

To conclude, let $S_{2*}\colon L^2(\Omega;F)\to L^2(\Omega;H^*)$ be the 
restriction of the adjoint of $S_2$ to $L^2(\Omega;F)$. Then 
$S_{2*}$ is a module map. Now
apply Lemma $\ref{bimod}$ to $S_1$ and $S_{2*}$. 
Let $\alpha\in L^\infty_\sigma(\Omega;B(E,H))$
and $\beta\in L^\infty_\sigma(\Omega;B(F,H^*))$ such that 
$S_1$ is equal to the multiplication by $\alpha$ and 
$S_{2*}$ is equal to the multiplication by $\beta$. 
Given any $e\in E$ and $f\in F$, we have
\begin{align*}
\int_{\Omega}\bigl\langle\bigl[\phi(t)](e),f\bigr\rangle\, 
x(t) y(t)\, \text{d}\mu(t)\, 
& =\,\bigl\langle T(x\otimes e),y\otimes f\bigr\rangle\\
& =\,\langle S_1(x\otimes e),S_{2*}(y\otimes f)\bigr\rangle\\
& =\,\int_\Omega\bigl\langle[\alpha(t)](e)\, x(t),
[\beta(t)](f)\, y(t)\bigr\rangle\, \text{d}\mu(t)\\
& =\,\int_\Omega\bigl\langle[\alpha(t)](e),
[\beta(t)](f)\bigr\rangle\,x(t) y(t) \,\text{d}\mu(t) 
\end{align*}
for any $x,y\in L^2$. Applying identification
between $H^*$ and $H$, this proves $(\ref{factor})$.
By construction, $\norm{\alpha}_\infty\leq C$ and 
$\norm{\beta}_\infty\leq 1$.
\end{proof}

\subsection{A special case: Schur multipliers} \label{Special}

Let $(\Omega_1,\mu_1)$,$(\Omega_2,\mu_2)$ and $(\Omega_3,\mu_3)$
be three separable measure spaces. We are going to  apply 
Theorem \ref{theofacto} with $(\Omega,\mu)=(\Omega_2,\mu_2)$,
$E=L^1(\Omega_1)$ and $F=L^1(\Omega_3)$.

To any $\phi\in L^\infty(\Omega_1\times\Omega_2\times\Omega_3)$, 
one may associate
$\widetilde{\phi}\in
L^\infty_\sigma\bigl(\Omega_2; B(L^1(\Omega_1),L^\infty(\Omega_3))\bigr)$ 
as follows.
For any $r\in L^1(\Omega_1)$,
\begin{equation}\label{tildephi}
\bigr[\widetilde{\phi}(t_2)\bigr](r)
= \int_{\Omega_1} \phi(t_1,t_2,\cdotp)\,r(t_1)\,\text{d}\mu_1(t_1),
\qquad t_2\in \Omega_2.
\end{equation}
According to the obvious identification
$$
L^\infty(\Omega_1\times\Omega_2\times\Omega_3) = 
L^\infty_\sigma\bigl(\Omega_2; L^\infty(\Omega_1\times\Omega_3)\bigr)
$$
and (\ref{IntForm}), the mapping $\phi\mapsto\widetilde{\phi}$ induces a
$w^*$-homeomorphic isometric identification
$$
L^\infty(\Omega_1\times\Omega_2\times\Omega_3)=
L^\infty_\sigma\bigl(\Omega_2; B(L^1(\Omega_1),L^\infty(\Omega_3))\bigr),
$$
By Remark \ref{WeakTensorisation}, the $w^*$-continuous contractive 
embedding of $\Gamma_2 
(L^1(\Omega_1),L^\infty(\Omega_3))$ into the space 
$B(L^1(\Omega_1),L^\infty(\Omega_3))$
induces a $w^*$-continuous contractive embedding
$$
L^\infty_\sigma\bigl(\Omega_2; \Gamma_2(L^1(\Omega_1),L^\infty(\Omega_3))\bigr)
\subset
L^\infty_\sigma\bigl(\Omega_2; B(L^1(\Omega_1),L^\infty(\Omega_3))\bigr).
$$
Combining with the preceding identification we obtain a further 
$w^*$-continuous contractive embedding
\begin{equation}\label{Meaning}
L^\infty_\sigma\bigl(\Omega_2; \Gamma_2(L^1(\Omega_1),L^\infty(\Omega_3))\bigr)
\subset L^\infty(\Omega_1\times\Omega_2\times\Omega_3).
\end{equation}
According to this, we will write
$\phi\in L^\infty_\sigma\bigl(\Omega_2; 
\Gamma_2(L^1(\Omega_1),L^\infty(\Omega_3))\bigr)$ 
when $\widetilde{\phi}$ actually belongs to that space.
In this case, for the sake of clarity, we let 
$$
\norm{\phi}_{\infty,\Gamma_2}
$$
denote its norm as an element of $L^\infty_\sigma\bigl(\Omega_2; 
\Gamma_2(L^1(\Omega_1),L^\infty(\Omega_3))\bigr)$. It is greater than or equal to
its norm as an element of $L^\infty(\Omega_1\times\Omega_2\times\Omega_3)$.

\begin{theorem}\label{factoL1} 
Let $\phi\in L^\infty(\Omega_1\times\Omega_2\times\Omega_3)$ and $C\geq 0$.
Then $\phi\in
L^\infty_\sigma\bigl(\Omega_2; \Gamma_2(L^1(\Omega_1),L^\infty(\Omega_3))\bigr)$ 
and $\norm{\phi}_{\infty,\Gamma_2}\leq C$ if and only if 
there 
exist a separable Hilbert space $H$ and two functions
$$
a \in L^\infty\bigl(\Omega_1\times\Omega_2;H\bigr)
\qquad\hbox{and}\qquad
b \in L^\infty\bigl(\Omega_2\times\Omega_3;H\bigr)
$$
such that $\norm{a}_\infty\norm{b}_\infty\leq C$ and
\begin{equation}\label{equal}
\phi(t_1,t_2,t_3) = \bigl\langle a(t_1,t_2), b(t_2,t_3)\bigr\rangle
\qquad\hbox{for a.e.}\ (t_1,t_2,t_3)\in \Omega_1\times\Omega_2\times\Omega_3.
\end{equation}
\end{theorem}

\begin{proof}
Assume that $\phi$ belongs to 
$L^\infty_\sigma\bigl(\Omega_2; \Gamma_2(L^1(\Omega_1),L^\infty(\Omega_3))\bigr)$, with 
$\norm{\phi}_{\infty,\Gamma_2}\leq C$.
According to Theorem \ref{theofacto}, there exist a Hilbert space $H$ and  two functions
$$
\alpha\in L^\infty_\sigma\bigl(\Omega_2;  B(L^1(\Omega_1),H)\bigr)
\qquad\hbox{and}\qquad
\beta\in L^\infty_\sigma\bigl(\Omega_2; B(L^1(\Omega_3),H)\bigr)
$$
such that for any $r_1\in L^1(\Omega_1)$ and $r_3\in L^1(\Omega_3)$,
\begin{equation}\label{THM2}
\bigl\langle [\widetilde{\phi}(t_2)](r_1),r_3\bigr\rangle\,=\,\bigl\langle
[\alpha(t_2)](r_1),[\beta(t_2)](r_3)\bigr\rangle\qquad
\hbox{for a.e.}\ t_2\in \Omega_2.
\end{equation}
By $(\ref{L1tensor})$, $(\ref{L1tensorcor})$ and (\ref{DP}) we have isometric identifications
\begin{align*}
L^\infty_\sigma\bigl(\Omega_2; B(L^1(\Omega_1),H)\bigr)\, &= 
L^\infty_\sigma\bigl(\Omega_2; (L^1(\Omega_1)\overset{\wedge}{\otimes} H^*)^*\bigr)\, \\
&= 
\bigl(L^1(\Omega_2)\overset{\wedge}{\otimes} L^1(\Omega_1)\overset{\wedge}{\otimes}H^*\bigr)^* \\ &=
L^1(\Omega_1\times\Omega_2; H^*)^* \\ &
= L^\infty_\sigma(\Omega_1\times\Omega_2; H).
\end{align*}
Moreover 
$L^\infty_\sigma(\Omega_1\times\Omega_2; H)=L^\infty(\Omega_1\times\Omega_2; H)$, see Remark \ref{RNP}.
Hence 
we finally have an isometric identification
$$
L^\infty_\sigma\bigl(\Omega_2; B(L^1(\Omega_1),H)\bigr)\,
=\, L^\infty(\Omega_1\times\Omega_2; H).
$$
Likewise we have
an isometric identification
$$
L^\infty_\sigma\bigl(\Omega_2; B(L^1(\Omega_3),H)\bigr)\,
=\, L^\infty(\Omega_2\times\Omega_3; H).
$$

Let $a\in L^\infty(\Omega_1\times\Omega_2; H)$ and $b\in L^\infty(\Omega_2\times\Omega_3; H)$
be corresponding to $\alpha$ and $\beta$ respectively in the above identifications.
Then for any $r_1\in L^1(\Omega_1)$,
$$
[\alpha(t_2)](r_1) =\,\int_{\Omega_1}a(t_1,t_2)\, r_1(t_1)\, \text{d}\mu_1(t_1)
\qquad 
\hbox{for a.e.}\ t_2\in\Omega_2.
$$
Likewise, for any $r_3\in L^1(\Omega_3)$,
$$
[\beta(t_2)](r_3) =\,\int_{\Omega_3} b(t_2,t_3)\, r_3(t_3)\, \text{d}\mu_3(t_3)
\qquad 
\hbox{for a.e.}\ t_2\in\Omega_2.
$$
Combining $(\ref{THM2})$ and $(\ref{tildephi})$ we deduce that for any $r_1\in L^1(\Omega_1)$ and
$r_3\in L^1(\Omega_3)$, we have
\begin{align*}
\int_{\Omega_1\times\Omega_3}\langle a(t_1,t_2),
b(t_2,t_3)\rangle\, & r_1(t_1)\,r_3(t_3)\, \text{d}\mu_1(t_1)\text{d}\mu_3(t_3)\\
& = 
\bigl\langle\bigl[\widetilde{\phi}(t_2)\bigr](r_1),r_3\bigr\rangle
\\ & = 
\int_{\Omega_1\times\Omega_3} \phi(t_1,t_2, t_3) r_1(t_1)\,r_3(t_3) 
\,\text{d}\mu_1(t_1)\text{d}\mu_3(t_3)
\end{align*}
for a.e. $t_2\in\Omega_2$.
This implies (\ref{equal}) and shows the `only if' part.

Assume conversely that (\ref{equal}) holds true for some $a$ in 
$L^\infty(\Omega_1\times\Omega_2;H)$ and some 
$b$ in $L^\infty(\Omega_1\times\Omega_2;H)$. Using the above identifications, we consider
$\alpha \in L^\infty_\sigma\bigl(\Omega_2; B(L^1(\Omega_1),H)\bigr)$ and
$\beta \in L^\infty_\sigma\bigl(\Omega_2; B(L^1(\Omega_3),H)\bigr)$ 
be corresponding
to $a$ and $b$, respectively. Then the above computations lead to (\ref{THM2}).
This identity means that for a.e. $t_2\in \Omega_2$, we have a Hilbert space factorisation 
$\widetilde{\phi}(t_2) = \beta(t_2)^*\alpha(t_2)$. This shows that 
$\phi\in L^\infty_\sigma\bigl(\Omega_2; \Gamma_2(L^1(\Omega_1),L^\infty(\Omega_3))\bigr)$, 
with $\norm{\phi}_{\infty,\Gamma_2}\leq
\norm{a}_\infty\norm{b}_\infty$.
\end{proof}

\vskip 1cm

\section{Characterization of $S^2\times S^2\to S^1$ boundedness}\label{mainresult}

Let $\mathcal{H}$ be a separable Hilbert space and let $A,B$ and $C$ be 
normal operators on $\mathcal{H}$.
Let $\lambda_A, \lambda_B$ and $\lambda_C$ 
be scalar-valued spectral measures associated with $A$, $B$ and $C$.
Recall the definition of the triple operator mapping $\Gamma^{A,B,C}$ from Theorem \ref{GammaABC}.
The purpose of this section is to characterize the functions 
$\phi\in L^\infty(\lambda_A\times\lambda_B\times\lambda_C)$
such that $\Gamma^{A,B,C}(\phi)$ maps $S^2(\mathcal{H})\times S^2(\mathcal{H})$
into $S^1(\mathcal{H})$.

We shall start with a factorization formula of independent 
interest. Let $\Gamma^{A,B}$ and $\Gamma^{B,C}$ be the double operator integral 
mappings associated respectively with $(A,B)$ and with $(B,C)$, 
see Proposition \ref{nintegrals}. As noted in Remark \ref{n=2},
$\Gamma^{A,B}$ and $\Gamma^{B,C}$ are $*$-representations.
Recall that they are $w^*$-continuous.

In the next statement we will consider the product $uv$ of a function
$u\in L^{\infty}(\lambda_A\times \lambda_B)$ and a function 
$v\in L^{\infty}(\lambda_B\times\lambda_C)$. The meaning is that we consider
$$
L^{\infty}(\lambda_A\times \lambda_B)\subset L^{\infty}(\lambda_A\times \lambda_B\times\lambda_C)
\qquad\hbox{and}\qquad
L^{\infty}(\lambda_B\times \lambda_C)\subset L^{\infty}(\lambda_A\times \lambda_B\times\lambda_C)
$$
in a canonical way and multiply $u$ and $v$ in this common bigger space.

\begin{lemma}\label{decomp}
Let $u\in L^{\infty}(\lambda_A\times \lambda_B)$ and 
$v\in L^{\infty}(\lambda_B\times\lambda_C)$. Then, for all $X,Y \in S^2(\mathcal{H})$, we have
$$
\Gamma^{A,B,C}(uv)(X,Y)=\Gamma^{A,B}(u)(X)\Gamma^{B,C}(v)(Y).
$$
\end{lemma}

\begin{proof}
Fix $X,Y \in S^2(\mathcal{H})$.
Let $u_1 \in L^{\infty}(\lambda_A), u_2,v_1 \in L^{\infty}(\lambda_B)$ and 
$v_2 \in L^{\infty}(\lambda_C)$. Consider $u=u_1 \otimes u_2 \in 
L^{\infty}(\lambda_A) \otimes L^{\infty}(\lambda_B)$
and $v=v_1 \otimes v_2 \in L^{\infty}(\lambda_B) \otimes L^{\infty}(\lambda_C)$.
Then we have $uv=u_1\otimes u_2v_1 \otimes v_2 \in L^{\infty}(\lambda_A) 
\otimes L^{\infty}(\lambda_B) \otimes L^{\infty}(\lambda_C)$. Therefore
\begin{align*}
\Gamma^{A,B,C}(uv)(X,Y) & = u_1(A)X(u_2v_1)(B)Yv_2(C)\\
& =u_1(A)Xu_2(B)v_1(B)Yv_2(C)\\ & =\Gamma^{A,B}(u)(X)\Gamma^{B,C}(v)(Y).
\end{align*}
Now, take $u\in L^{\infty}(\lambda_A\times \lambda_B)$ and $v\in L^{\infty}(\lambda_B\times \lambda_C)$.
Let $(u_i)_{i}$ and  $(v_j)_{j}$ be two nets in $L^{\infty}(\lambda_A) \otimes L^{\infty}(\lambda_B)$ 
and $L^{\infty}(\lambda_B) \otimes L^{\infty}(\lambda_C)$ respectively,
converging to $u$ and $v$ in the $w^*$-topology. By linearity, the previous
calculation implies that for all $i,j$,
$$
\Gamma^{A,B,C}(u_iv_j)(X,Y)=\Gamma^{A,B}(u_i)(X)\Gamma^{B,C}(v_j)(Y).
$$

Take $Z\in S^2(\mathcal{H})$ and fix $j$. Since $\Gamma^{B,C}(v_j)(Y)Z$
belongs to $S^2(\mathcal{H})$ we have
\begin{align*}
\underset{i}{\lim} \ \text{tr}(\Gamma^{A,B}(u_i)(X)\Gamma^{B,C}(v_j)(Y)Z)
& =\text{tr}(\Gamma^{A,B}(u)(X)\Gamma^{B,C}(v_j)(Y)Z) \\
& = \text{tr}(\Gamma^{B,C}(v_j)(Y)Z\Gamma^{A,B}(u)(X))
\end{align*}
by the $w^*$-continuity of $\Gamma^{A,B}$.
Similarly, since $Z\Gamma^{A,B}(u)(X) \in S^2(\mathcal{H})$, 
the $w^*$-continuity of $\Gamma^{B,C}$ implies that
\begin{align*}
\underset{j}{\lim} \ \text{tr}(\Gamma^{B,C}(v_j)(Y)Z\Gamma^{A,B}(u)(X))
& =\text{tr}(\Gamma^{B,C}(v)(Y)Z\Gamma^{A,B}(u)(X)) \\
& = \text{tr}(\Gamma^{A,B}(u)(X)\Gamma^{B,C}(v)(Y)Z).
\end{align*}
On the other hand, $(u_iv_j)_i$ $w^*$-converges to $uv_j$ for any fixed $j$ and 
$(uv_j)_j$ $w^*$-converges to $uv$ in $L^{\infty}(\lambda_A\times \lambda_B \times \lambda_C)$.  
Hence the $w^*$-continuity of $\Gamma^{A,B,C}$ implies that
\begin{align*}
\underset{j}{\lim} \, \underset{i}{\lim} \ \text{tr}(\Gamma^{A,B,C}(u_iv_j)(X,Y)Z)  
& = \underset{j}{\lim} \ \text{tr}(\Gamma^{A,B,C}(uv_j)(X,Y)Z) \\
& = \text{tr}(\Gamma^{A,B,C}(uv)(X,Y)Z).
\end{align*}
Thus, for all $Z\in S^2(\mathcal{H})$,
$$
\text{tr}(\Gamma^{A,B}(u)(X)\Gamma^{B,C}(v)(Y)Z) = \text{tr}(\Gamma^{A,B,C}(uv)(X,Y)Z),
$$
which implies that $\Gamma^{A,B,C}(uv)=\Gamma^{A,B}(u)(X)\Gamma^{B,C}(v)(Y).$
\end{proof}

The next theorem is our main result. 
It should be regarded as an extension of \cite[Corollary 8]{CLPST1} to the measurable 
setting. In the latter statement one considers a matrix $M=\{m_{ikj}\}_{i,k,j\geq 1}$ 
and it is implicitly shown that the bilinear Schur multiplier associated with $M$ maps
$S^2\times S^2$ into $S^1$ if and only if $M$ belongs to 
$\ell^\infty\bigl(\Gamma_2(\ell^1,\ell^\infty)\bigr)$.
In the current situation, matrices are replaced by functions. The scheme of 
proof of Theorem \ref{main} is similar to the one of \cite[Corollary 8]{CLPST1}
but requires various additional tools.

\begin{theorem}\label{main}
Let $\mathcal{H}$ be a separable Hilbert space, let $A,B$ and $C$ be normal operators on $\mathcal{H}$ 
and let $\phi\in L^{\infty}(\lambda_A \times 
\lambda_B \times \lambda_C)$. The following are equivalent :
\begin{enumerate}
\item[(i)] $\Gamma^{A,B,C}(\phi) \in B_2(S^2(\mathcal{H}) 
\times S^2(\mathcal{H}), S^1(\mathcal{H})).$
\item[(ii)] There exist a separable Hilbert space $H$ and two functions
$$
a\in L^{\infty}(\lambda_A \times \lambda_B ; H) \qquad
\text{and} \qquad
b\in L^{\infty}(\lambda_B\times \lambda_C ; H)
$$
such that 
$$
\phi(t_1,t_2,t_3)= \left\langle a(t_1,t_2),b(t_2,t_3) \right\rangle
$$
for a.e. $(t_1,t_2,t_3) \in \sigma(A) \times \sigma(B) \times \sigma(C).$
\end{enumerate}
In this case, 
\begin{equation}\label{AllEqual}
\bignorm{\Gamma^{A,B,C}(\phi)\colon S^2(\H)\times S^2(\H)
\longrightarrow S^1(\H)}=\inf\bigl\{\norm{a}_\infty\norm{b}_\infty\bigr\},
\end{equation}
where the infimum runs over all pairs $(a,b)$ satisfying (ii).
\end{theorem}

\begin{proof}

\smallskip\noindent
\underline{(ii) $\Rightarrow$ (i)}: 
Assume (ii) and let 
$(\epsilon_k)_{k\in \mathbb{N}}$ be a Hilbertian basis of $H$. 
For any $k\in\Ndb$, define 
$$
a_k=\left\langle a,\epsilon_k \right\rangle \in L^{\infty}(\lambda_A \times \lambda_B) 
\qquad \text{and} \qquad b_k=\left\langle b,\epsilon_k 
\right\rangle \in L^{\infty}(\lambda_B\times \lambda_C).
$$
We set 
$$
\vert a \vert=\Bigl(\sum_n |a_k|^2\Bigr)^{\frac12};
$$
this function belongs to  $L^{\infty}(\lambda_A \times \lambda_B) $ and we have
$\norm{a}_\infty = \norm{\vert a \vert}_\infty$.

Let $X\in S^2(\mathcal{H})$. Since $\Gamma^{A,B}$ is a $w^*$-continuous $*$-representation, we have
\begin{align*}
\sum_k \| \Gamma^{A,B}(a_k)(X)\|_2^2
& = \sum_k \left\langle \Gamma^{A,B}(a_k)(X), \Gamma^{A,B}(a_k)(X) \right\rangle \\
& = \sum_n \left\langle \Gamma^{A,B}(\overline{a_k})\Gamma^{A,B}(a_k)(X), X \right\rangle \\
& = \left\langle \Gamma^{A,B}(\vert a \vert^2)(X), X \right\rangle \\
& \leq \|\vert a \vert^2\|_\infty \|X\|_2^2 = \|a\|_\infty^2 \|X\|_2^2.
\end{align*}
We prove similarly that if $Y\in S^2(\mathcal{H})$, then
$$
\sum_n \| \Gamma^{B,C}(\overline{b_k})(Y)\|_2^2 \leq \|b\|_\infty^2 \|Y\|_2^2.
$$
Consequently, for all $X,Y \in S^2(\mathcal{H})$, we have the inequalities
\begin{align*}
\sum_k \| \Gamma^{A,B}(a_k)(X)\Gamma^{B,C}(\overline{b_k})(Y) \|_1
& \leq \sum_k \| \Gamma^{A,B}(a_k)(X) \|_2 \| \Gamma^{B,C}(\overline{b_k})(Y) \|_2 \\
& \leq \Bigl( \sum_k \| \Gamma^{A,B}(a_k)(X) \|_2^2\Bigr)^{1/2}
\Bigl( \sum_k \| \Gamma^{B,C}(\overline{b_k})(Y) \|_2^2\Bigr)^{1/2} \\
& \leq \|a\|_\infty\|b\|_\infty\|X\|_2 \|Y\|_2.
\end{align*}
Therefore, we can define a bounded bilinear map 
$$
\Theta\colon S^2(\mathcal{H}) \times S^2(\mathcal{H}) \longrightarrow  S^1(\mathcal{H})
$$
by 
$$
\Theta(X,Y) = \sum_{k=1}^{\infty} \Gamma^{A,B}(a_k)(X)\Gamma^{B,C}(\overline{b_k})(Y),
\qquad X,Y\in S^2(\mathcal{H}),
$$
and we have 
\begin{equation}\label{Tab}
\norm{\Theta}\leq \norm{a}_\infty\norm{b}_\infty.
\end{equation}

We claim that 
$$
\Gamma^{A,B,C}(\phi) = \Theta.
$$
To check this, consider 
$$
\widetilde{a_n} = \sum_{k=0}^n a_k\otimes 
\epsilon_k \qquad\text{and} \qquad\widetilde{b_n} = \sum_{k=0}^n b_k\otimes 
\epsilon_k
$$
for any $n\in \mathbb{N}$.
Then we set
$$
\phi_n(t_1,t_2,t_3) = \bigl\langle \widetilde{a_n} (t_1,t_2), \widetilde{b_n}(t_2,t_3)\bigr\rangle
=\sum_{k=0}^n a_k(t_1,t_2)\overline{b_k(t_2,t_3)}.
$$
Fix $X,Y \in S^2(\mathcal{H})$. 
We have $\Gamma^{A,B,C}(\phi_n) = \sum_{k=0}^n \Gamma^{A,B,C} (a_k\overline{b_k})\,$ hence by Lemma 
\ref{decomp}, 
$$
\Gamma^{A,B,C}(\phi_n)(X,Y) = \sum_{k=0}^n \Gamma^{A,B}(a_k)(X)\Gamma^{B,C}(\overline{b_k})(Y).
$$
Consequently,
$$
\Gamma^{A,B,C}(\phi_n)(X,Y) 
\underset{n\rightarrow +\infty}{\longrightarrow} \Theta(X,Y) \ \ 
\text{in} \ S^1(\mathcal{H}).
$$
Moreover $\phi_n\to\phi$ a.e. and 
$(\phi_n)_n$ is bounded in $L^{\infty}(\lambda_A \times \lambda_B \times \lambda_C)$.
Indeed,
$$
\bigl\vert 
\phi_n(t_1,t_2,t_3) \bigr\vert  \leq \Bigl(\sum_{k=0}^n \vert a_k(t_1,t_2)\vert^2\Bigr)^{\frac12}
\Bigl(\sum_{k=0}^n \vert 
b_k(t_2,t_3)\vert^2\Bigr)^{\frac12}
\leq \norm{a}_\infty\norm{b}_\infty.
$$
Hence by Lebesgue's dominated convergence theorem, $w^*$-$\underset{n\rightarrow +\infty}{\lim}\phi_n = \phi$. 
The $w^*$-continuity of $\Gamma^{A,B,C}$ implies that
$$
\Gamma^{A,B,C}(\phi_n)(X,Y) \underset{n\rightarrow +\infty}{\longrightarrow} 
\Gamma^{A,B,C}(\phi)(X,Y)
$$
weakly in $S^2(\mathcal{H})$.
We conclude that $\Gamma^{A,B,C}(\phi)(X,Y) = \Theta(X,Y)$.

This shows (i). Furthermore  (\ref{Tab})  yields
\begin{equation}\label{Ineq1}
\bignorm{\Gamma^{A,B,C}(\phi)\colon S^2(\H)\times S^2(\H)\longrightarrow S^1(\H)} \leq \norm{a}_\infty\norm{b}_\infty.
\end{equation}

\smallskip\noindent
\underline{(i) $\Rightarrow$ (ii)}: 
As in Subsection 
\ref{Op-to-Fu}, we consider the triple integral mappings $\Lambda(\phi)$ 
in the case when
$(\Omega_1,\mu_1)=(\sigma(C),\lambda_C)$,
$(\Omega_2,\mu_2)=(\sigma(B),\lambda_B)$
and
$(\Omega_3,\mu_3)=(\sigma(A),\lambda_A)$. Note that these measurable spaces
are separable.

Assume (i) and apply Proposition \ref{Connection}, which connects
$\Gamma^{A,B,C}(\phi)$ to $\Lambda(\phi)$. Let 
$$
X\in S^2(L^2(\lambda_B), L^2(\lambda_A))
\qquad\hbox{and}\qquad 
Y \in S^2(L^2(\lambda_C), L^2(\lambda_B)).
$$
By (\ref{subspace}), we have
\begin{align*}
\|\Lambda(\phi)(X,Y)\|_1 
& = \|\rho_A^{-1} \circ \Gamma^{A,B,C}(\phi)(\widetilde{X},\widetilde{Y}) \circ \rho_C\|_1 \\
& \leq  \|\Gamma^{A,B,C}(\phi)(\widetilde{X},\widetilde{Y})\|_1 \\
& \leq \bignorm{\Gamma^{A,B,C}(\phi)\colon S^2\times S^2\to S^1}\| X \|_2 \| Y \|_2,
\end{align*}
since $\| \widetilde{X} \|_2=\norm{X}_2$ and $\| \widetilde{Y} \|_2=\norm{Y}_2$.
This shows that $\Lambda(\phi)$ maps $S^2\times S^2$ into $S^1$, with
\begin{equation}\label{comparison2}
\bignorm{\Lambda(\phi)\colon S^2\times S^2\to S^1}\leq 
\bignorm{\Gamma^{A,B,C}(\phi)\colon S^2\times S^2\to S^1}.
\end{equation}

We now extend the proof of \cite[Corollary 8]{CLPST1} to get a Hilbert space factorization.
For convenience we write $H_1 = L^2(\lambda_A), H_2 = L^2(\lambda_B)$ and $H_3 = L^2(\lambda_C)$. 
Each of these spaces naturally identifies with its conjugate space, hence we will
not use conjugation bars as we had to do in Subsection \ref{Hilbert}.

We have just proved above that $\Lambda(\phi)$ extends to a bounded bilinear map 
$S^2(H_2, H_1)\times S^2(H_3, H_2)$ into $S^1(H_3, H_1)$. According to the identification
$$
B_2(S^2 \times S^2,S^2)=B(S^2 \overset{\wedge}{\otimes} S^2,S^2),
$$
provided by (\ref{biliproj}), it can be also regarded as a bounded 
linear operator from the projective tensor product $S^2(H_2,H_1) \overset{\wedge}{\otimes} S^2(H_3,H_2)$
into $S^1(H_3,H_1)$. 
By
(\ref{S2=L2}), we may naturally identify $S^2(H_3, H_2)$ and  $S^2(H_2, H_1)$ with
the Bochner spaces $L^2(\lambda_B;H_3)$ and $L^2(\lambda_B;H_1)$, respectively. We
may therefore regard $\Lambda(\phi)$
as a bounded linear operator 
$$
\Lambda(\phi)\colon L^2(\lambda_B;H_1)\overset{\wedge}{\otimes} L^2(\lambda_B;H_3) 
\longrightarrow S^1(H_3,H_1).
$$
Property (\ref{dualproj}) and Hilbert space self-duality provide a natural
isometric identification
$$
\bigl(L^2(\lambda_B;H_1)\overset{\wedge}{\otimes} L^2(\lambda_B;H_3) \bigr)^*\,
=\, B\bigl(L^2(\lambda_B;H_1), L^2(\lambda_B;H_3)\bigr).
$$
We further have $S^1(H_3,H_1)^*=B(H_1,H_3)$, by (\ref{Dual-S1}). We now let
$$
v\colon B(H_1,H_3)\longrightarrow B\bigl(L^2(\lambda_B;H_1), L^2(\lambda_B;H_3)\bigr)
$$
be the adjoint of $\Lambda(\phi)$ through these identifications.

According to (\ref{Multiplication}) and
(\ref{Multiplication2}), we have an isometric embedding
$$
L^\infty_\sigma\bigl(\lambda_B;B(H_1,H_3)\bigr)\subset B\bigl(L^2(\lambda_B;H_1), L^2(\lambda_B;H_3)\bigr)
$$
obtained by identifying any $\Delta\in L^\infty_\sigma\bigl(\lambda_B;B(H_1,H_3)\bigr)$ with 
the multiplication operator $M_\Delta$. It is easy to check (left to the reader)
that this embedding is $w^*$-continuous. Hence we may regard the dual space  
$L^\infty_\sigma\bigl(\lambda_B;B(H_1,H_3)\bigr)$ as a $w^*$-closed subspace of 
$B\bigl(L^2(\lambda_B;H_1), L^2(\lambda_B;H_3)\bigr)$.
We aim at showing (\ref{v}) below.

Let $\xi\in H_1$ and $\eta\in H_3$, and consider 
$\xi \otimes \eta$ as an element of $B(H_1,H_3)$.
Take any 
$c \in H_1$, $c',d' \in L^2(\lambda_B)$ and $d\in H_3$, then
regard $c'\otimes c$ as an element of $L^2(\lambda_B;H_1)$
and $d'\otimes d$ as an element of $L^2(\lambda_B;H_3)$.
We have
\begin{align*}
\bigl\langle \bigl[v(\xi & \otimes \eta)\bigr](c'\otimes c), d' \otimes d\bigr\rangle_{L^2(\lambda_B;H_3),
L^2(\lambda_B;H_3)} \\
& = \bigl\langle \xi \otimes \eta,\Lambda(\phi)\bigl[(c' \otimes c) \otimes (d'\otimes d)\bigr] 
\bigr\rangle_{B(H_1,H_3),S^1(H_3,H_1)}\\
& = \int_{\sigma(A) \times \sigma(B) \times \sigma(C)}
\phi(t_1,t_2,t_3) \xi(t_1) \eta(t_3)
c'(t_2) d'(t_2) c(t_1)  d(t_3) \,
\text{d}\lambda_A(t_1)\text{d}\lambda_B(t_2)\text{d}\lambda_C(t_3)\,.
\end{align*}
It readily follows from this formula that for any $\theta\in L^\infty(\lambda_B)$,
$$
\bigl\langle \bigl[v(\xi \otimes \eta)\bigr](\theta c'\otimes c), d' \otimes d\bigr\rangle\,
=\, \bigl\langle \bigl[v(\xi \otimes \eta)\bigr](c'\otimes c), \theta d' \otimes d\bigr\rangle.
$$
Since $L^2(\lambda_B)\otimes H_1$ and $L^2(\lambda_B)\otimes H_3$
are dense in $L^2(\lambda_B;H_1)$ and $L^2(\lambda_B;H_3)$, respectively, this implies that 
$[v(\xi \otimes \eta)] (\theta \varphi) =\theta [v(\xi  \otimes \eta)] (\varphi)$ for any 
$\varphi\in L^2(\lambda_B;H_1)$ and any $\theta\in L^\infty(\lambda_B)$. By Lemma \ref{bimod}, this implies that
$v(\xi  \otimes \eta)$ belongs to $L^{\infty}_{\sigma}\bigl(\lambda_B, B(H_1,H_3)\bigr)$.

Since $v$ is $w^*$-continuous and $H_1\otimes H_3$ is 
$w^*$-dense in $B(H_1,H_3)$, we deduce that
\begin{equation}\label{v}
v\bigl(B(H_1,H_3)\bigr) 
\subset L^{\infty}_{\sigma}(\lambda_B; B(H_1,H_3)).
\end{equation}

Consider now the restriction $v_0=v_{|\mathcal{K}(H_1,H_3)}$ of $v$ 
to the subspace $\mathcal{K}(H_1,H_3)$ of compact operators 
from $H_1$ into $H_3$. By 
(\ref{v}), we may write
$$
v_0\colon \mathcal{K}(H_1,H_3)\longrightarrow L^\infty_\sigma \bigl(\lambda_B;
B(H_1,H_3)\bigr).
$$
Corollary \ref{CoroDuality} provides an identification
$$
B\bigl(\mathcal{K}(H_1,H_3),L^\infty_\sigma(\lambda_B;
B(H_1,H_3))\bigr) = 
L^\infty_\sigma\bigl(\lambda_B;B(\mathcal{K}(H_1,H_3),
B(H_1,H_3))\bigr).
$$
Let $\widetilde{\phi}\in L^\infty_\sigma\bigl(\lambda_B;B(\mathcal{K}(H_1,H_3),
B(H_1,H_3))\bigr)$ be corresponding to
$v_0$ in this identification. Then by the preceding computation
we have that for any $c, \xi \in H_1$ and $d,\eta\in H_3$,
$$
\bigl\langle \bigl[\widetilde{\phi}(t_2)\bigr](\xi\otimes\eta), d\otimes c\bigr\rangle
= \int_{\sigma(A)\times\sigma(C)}
\phi(t_1,t_2,t_3)\xi(t_1)\eta(t_3)c(t_1)d(t_3)\,
\text{d}\lambda_A(t_1)\text{d}\lambda_C(t_3)
$$
for a.e. $t_2$ in $\sigma(B)$.

Following Subsection \ref{Schurmulti}, for any $J\in L^2(\lambda_A\times\lambda_C)$,
we let $X_J\in S^2(H_1,H_3)$ be the Hilbert-Schmidt
operator with kernel $J$. Then the above formula shows that 
for $J=\xi\otimes\eta$, we have
\begin{equation}\label{Connexion}
\bigl[\widetilde{\phi}(t_2)\bigr](X_J) =
 X_{\phi(\cdotp,t_2,\cdotp)J}\qquad \hbox{for a.e.}\, t_2.
\end{equation}
By density of $H_1\otimes H_3$ in $L^2(\lambda_A\times\lambda_C)$, we deduce that
(\ref{Connexion}) holds true for any $J\in L^2(\lambda_A\times\lambda_C)$.
This means that for a.e. $t_2$, $\phi(\cdotp,t_2,\cdotp)$, regarded as an element 
of $L^\infty(\lambda_A\times\lambda_C)$, is a measurable Schur
multiplier, whose corresponding operator is 
$$
\widetilde{\phi}(t_2) = R_{\phi(\cdotp,t_2,\cdotp)} \colon 
\mathcal{K}(L^2(\lambda_A),L^2(\lambda_C)) \longrightarrow
B(L^2(\lambda_A),L^2(\lambda_C)).
$$
This shows two things. First, $\widetilde{\phi}$ belongs to
$L^\infty_\sigma\bigr(\lambda_B; \Gamma_2(L^1(\lambda_A), L^\infty(\lambda_C))\bigr)$
regarded as a subspace of 
$L^\infty_\sigma\bigl(\lambda_B;B(\mathcal{K}(H_1,H_3),
B(H_1,H_3))\bigr)$, by (\ref{L-infty-inclusion}). Second,
the element of $L^\infty(\lambda_A\times\lambda_B\times\lambda_C)$ 
corresponding to $\widetilde{\phi}$ through the inclusion
(\ref{Meaning}) is the function $\phi$ itself. Thus we have proved that
$\phi\in L^\infty_\sigma\bigr(\lambda_B; \Gamma_2(L^1(\lambda_A), 
L^\infty(\lambda_C))\bigr)$. Further the above reasoning shows (using the notation
$\norm{\,\cdotp\,}_{\infty,\Gamma_2}$ introduced after (\ref{Meaning})) that
$$
\norm{\phi}_{\infty,\Gamma_2}\leq
\bignorm{\Lambda(\phi)\colon S^2 \times S^2 \to S^1}.
$$
According to (\ref{comparison2}), this implies that
$$
\norm{\phi}_{\infty,\Gamma_2}\leq
\bignorm{\Gamma^{A,B,C}(\phi)\colon S^2 \times S^2 \to S^1}.
$$
Now applying Theorem \ref{factoL1} yields (ii), with 
$\norm{a}_\infty\norm{b}_\infty\leq
\bignorm{\Gamma^{A,B,C}(\phi)\colon S^2 \times S^2 \to S^1}$.
\end{proof}

Theorem \ref{main} extends to the framework of triple operator
integrals associated with functions as defined in Subsection \ref{Functions}. 
With similar proofs as above, we obtain the following.

\begin{theorem}
Let $(\Omega_1, \mu_1), (\Omega_2, \mu_2)$ and $(\Omega_3, \mu_3)$ be 
three separable measure spaces, and let $\phi \in 
L^{\infty}(\Omega_1 \times \Omega_2 \times \Omega_3)$. Then
$\Lambda(\phi)$ extends to a bounded bilinear map
$$
\Lambda(\phi) : S^2(L^2(\Omega_2), L^2(\Omega_3)) \times 
S^2(L^2(\Omega_1), L^2(\Omega_2)) \rightarrow S^1(L^2(\Omega_1), L^2(\Omega_3))$$
if and only if there exist a separable Hilbert space $H$ and two functions
$$
a\in L^{\infty}(\Omega_1 \times \Omega_2 ; H) \qquad
\text{and} \qquad
b\in L^{\infty}(\Omega_2\times \Omega_3 ; H)
$$
such that 
$$
\phi(t_1,t_2,t_3)= \left\langle a(t_1,t_2), b(t_2,t_3) \right\rangle
$$
for a.e. $(t_1,t_2,t_3) \in \Omega_1 \times \Omega_2 \times \Omega_3.$

In this case, 
\begin{equation}
\bignorm{\Lambda(\phi) \colon  S^2 \times S^2 \rightarrow S^1 }=\inf\bigl\{\norm{a}_\infty\norm{b}_\infty\bigr\},
\end{equation}
where the infimum runs over all pairs $(a,b)$ verifying the above factorization property.
\end{theorem}

\vskip 1cm

\section{Additional comments}\label{add}

In this last section, we explain connections between our theorems and previous 
results in this area. We first show that Peller's Theorem from \cite{Peller1985}
(mentioned in the Introduction) is a direct consequence of Theorem \ref{main}. 
With the terminology of the present paper, Peller's Theorem can be stated as follows.

\begin{theorem}\label{DOI} (Peller \cite{Peller1985})
Let $A,B$ be normal operators on a separable Hilbert space
$\H$ and let $\lambda_A$ and $\lambda_B$ be scalar-valued 
spectral measures for $A$ and $B$. For any $\psi\in L^\infty(\lambda_A\times
\lambda_B)$, the following are equivalent. 

\begin{itemize}
\item [(i)] The double operator integral mapping 
$\Gamma^{A,B}(\psi)$ extends to a bounded map from $S^1(\H)$
into itself.
\item [(ii)] There exist a separable Hilbert space $H$ and two functions
$a\in L^\infty(\lambda_A;H)$ and $b\in L^\infty(\lambda_B;H)$ such that
\begin{equation}\label{FactorPeller}
\psi(s,t) = \langle a(s), b(t)\rangle 
\end{equation}
for a.e. $(s,t)\in \sigma(A)\times\sigma(B)$.
\end{itemize}
In this case, 
$$
\bignorm{\Gamma^{A,B}(\psi)\colon S^1(\H)\longrightarrow S^1(\H)}
=\inf\bigl\{\norm{a}_\infty\norm{b}_\infty\bigr\},
$$ 
where the infimum  runs over all pairs $(a,b)$ of functions such that 
(\ref{FactorPeller}) holds true.
\end{theorem}

\begin{proof}
Consider $A,B$ as above and take an auxiliary normal operator $C$ on $\H$ (this may be 
the identity map), with a scalar-valued 
spectral measure $\lambda_C$. For any $\psi\in L^\infty(\lambda_A\times\lambda_B)$, set 
$$
\widetilde{\psi} = 
\psi\otimes 1\in L^\infty(\lambda_A\times\lambda_B)\otimes L^\infty(\lambda_C)\subset 
L^\infty(\lambda_A\times\lambda_C\times\lambda_B).
$$
We claim that for any $X,Y\in S^2(\H)$, 
\begin{equation}\label{TOI-to-DOI}
\Gamma^{A,C,B}(\widetilde{\psi})(X,Y) = \Gamma^{A,B}(\psi)(XY).
\end{equation}
Indeed for any $f_1\in L^\infty(\lambda_A)$ and $f_2\in L^\infty(\lambda_B)$, 
and for any $X,Y\in S^2(\H)$, we have
$$
\Gamma^{A,C,B}(f_1\otimes 1\otimes f_2)(X,Y) = f_1(A)XYf_2(B).
$$
Hence by linearity, (\ref{TOI-to-DOI}) holds true for any 
$\psi\in L^\infty(\lambda_A)\otimes L^\infty(\lambda_B)$. By the $w^*$-continuity
of $\Gamma^{A,C,B}$ and of $\Gamma^{A,B}$, this identity
holds as well for any $\psi\in L^\infty(\lambda_A\times\lambda_B)$.

We have $\norm{XY}_1\leq \norm{X}_2\norm{Y}_2$ for any 
$X,Y\in S^2(\H)$ and conversely, for any $Z\in S^1(\H)$, there
exist $X,Y$ in $S^2(\H)$ such that $XY=Z$ and $\norm{X}_2\norm{Y}_2
=\norm{Z}_1$. Thus given any $\psi\in L^\infty(\lambda_A\times\lambda_B)$,
it follows from (\ref{TOI-to-DOI}) that 
$\Gamma^{A,C,B}(\widetilde{\psi})$ maps 
$S^2(\H)\times S^2(\H)$ into $S^1(\H)$ if and only if 
$\Gamma^{A,B}(\psi)$ maps $S^1(\H)$ into $S^1(\H)$ and moreover,
$$
\bignorm{\Gamma^{A,C,B}(\widetilde{\psi})\colon S^2(\H)\times S^2(\H)
\longrightarrow S^1(\H)}
=\bignorm{\Gamma^{A,B}(\psi)\colon S^1(\H)
\longrightarrow S^1(\H)}.
$$

On the other hand, $\widetilde{\psi}$ satisfies condition (ii) from Theorem \ref{main}
if and only if $\psi$ satisfies condition (ii) from Theorem \ref{DOI}.

The result therefore follows from Theorem \ref{main}.
\end{proof}

\begin{remark}\label{GT} 
In this remark, 
we discuss another formulation of Peller's Theorem. 
Let $E,F$ be Banach spaces. 
A bounded map $u\colon E\to F^*$ is called integral if there 
exist a probability measure space $(\Sigma,\nu)$ and two
bounded maps $\alpha\colon E\to L^\infty(\nu)$
and $\beta\colon F\to L^\infty(\nu)$ such that
\begin{equation}\label{Integral}
\langle u(x),y\rangle = \int_{\Sigma}\bigl[\alpha(x)](\omega)
[\beta(y)](\omega)\,\text{d}\nu(\omega),\qquad x\in E,y\in F.
\end{equation}
Let $\I(E,F^*)$ denote the space of all such operators and set
$I(u)=\inf\{\norm{\alpha}\norm{\beta}\}$, where
the infimum runs over all such factorizations. 
Then $I(\cdotp)$ is a norm on $\I(E,F^*)$ and the latter is a Banach
space. Moreover (\ref{dualproj}) induces an
isometric identification
$$(E \overset{\vee}{\otimes} F)^* = \I(E,F^*),
$$ 
where $\overset{\vee}{\otimes}$ is the injective tensor product.
We refer e.g. to \cite[Chapter VIII, Theorems 5 $\&$ 9]{Diestel}
for these definitions and properties.

Grothendieck's Inequality on tensor products implies that 
for any measure spaces
$(\Omega_1,\mu_1)$ and $(\Omega_2,\mu_2)$, and for any $z\in L^1(\mu_1)\otimes
L^1(\mu_2)$, we have $\norm{z}_\vee\leq \gamma_2^*(z)\leq K\norm{z}_\vee$, where
$K$ is a universal constant (see e.g. \cite[Section 3]{PisierGroth}).
Equivalently, 
$$
L^1(\mu_1)\hat{\otimes}_{\gamma_2^*} L^1(\mu_2)\approx
L^1(\mu_1)\overset{\vee}{\otimes} L^1(\mu_2)
$$
$K$-isomorphically.
Passing to duals, this yields a $w^*$-homeomorphic 
$K$-isomorphism 
\begin{equation}\label{GT-Ineq}
\Gamma_2(L^1(\mu_1),L^\infty(\mu_2))\approx  
\I(L^1(\mu_1),L^\infty(\mu_2)).
\end{equation}

Let $A,B$ as in Theorem \ref{DOI}, let $\psi\in L^\infty(\lambda_A\times
\lambda_B)$ and let $u_\psi\colon L^1(\lambda_A)\to
L^\infty(\lambda_B)$ be the bounded map associated to $\psi$
(see (\ref{IntForm})). Condition (ii) from Theorem \ref{DOI} means
that $u_\psi\in \Gamma_2(L^1(\lambda_A),
L^\infty(\lambda_B))$. 
Hence in Theorem \ref{DOI} above, condition
(i) is also equivalent to :
\begin{itemize}
\item [(iii)] {\it The operator $u_\psi$ belongs to $\I(L^1(\lambda_A),
L^\infty(\lambda_B))$.}
\end{itemize}
Further it is easy to deduce from the above definition 
of integral operators (see (\ref{Integral})) that 
the above property (iii) is formally equivalent
to :
\begin{itemize}
\item [(iv)] {\it There exist a probability measure space $(\Sigma,\nu)$ and two
functions $a\in L^\infty(\lambda_A\times\nu)$ and $b\in L^\infty(\lambda_B\times\nu)$ 
such that
$$
\psi(s,t) = \int_{\Sigma} a(s,\omega)b(t,\omega)\,\text{d}\nu(\omega),
\qquad a.e.\hbox{-}(s,t).
$$}
\end{itemize}
The equivalence between (i) 
and (iv) is stated in \cite{Peller1985, Peller2015}, and also in 
\cite{Hiai,J} to which we refer for various proofs. It follows from this analysis that 
if condition (i) from Theorem \ref{DOI} holds true, then the above factorization 
(iv) can be achieved with $(a,b)$ 
satisfying 
$$
\norm{a}_\infty\norm{b}_\infty\leq K \bignorm{\Gamma^{A,B}(\psi)\colon S^1(\H)\longrightarrow S^1(\H)}.
$$
Conversely if (iv) holds true, then $\norm{\Gamma^{A,B}(\psi)\colon S^1(\H)\longrightarrow S^1(\H)}
\leq\norm{a}_\infty\norm{b}_\infty$.

We note that the original paper \cite{Peller1985} makes use of
Grothendieck's Inequality to establish 
Theorem \ref{DOI}. Our approach shows that this can be avoided
and that Grothendieck' Inequality is useful only to 
establish the equivalence of (iv) with (i).

Let us now come back to Theorem \ref{main}. Let $A,B,C$ and $\phi$
as in this theorem. It follows from the proof of Theorem \ref{main}
that the conditions (i)-(ii) in this theorem are equivalent
to the fact that $\phi \in L^{\infty}_{\sigma}(\lambda_B; 
\Gamma_2(L^1(\lambda_A), L^{\infty}(\lambda_C)))$. Hence according to
(\ref{GT-Ineq}), the conditions 
of Theorem \ref{main} are also equivalent to 
\begin{equation}\label{Main+GT}
\phi \in L^{\infty}_{\sigma}(\lambda_B; 
\I(L^1(\lambda_A), L^{\infty}(\lambda_C))).
\end{equation}
It is a natural question whether this implies the existence of a 
probability measure space $(\Sigma,\nu)$ and two
functions $a\in L^\infty(\lambda_A\times\lambda_B\times\nu)$ and 
$b\in L^\infty(\lambda_B\times\lambda_C\times\nu)$ 
such that $\phi(t_1,t_2,t_3)
= \int_{\Sigma} a(t_1,t_2,\omega)b(t_2,t_3,\omega)\,\text{d}\nu(\omega)$
for a.e. $(t_1,t_2,t_3)$. However we haven't been able to establish this yet.
\end{remark}

We now turn to connections between Theorem \ref{GammaABC} or
Proposition \ref{nintegrals} and the constructions of multiple operator 
integrals from \cite{Peller2006} and \cite{ACDS}.

Let $A_1,\ldots, A_n$ be normal operators on a separable Hilbert space 
$\H$. Throughout we use the notations of Proposition \ref{nintegrals}.
Let $(\Sigma,d\mu)$ be a $\sigma$-finite measure space and, for any $i=1,\ldots,n$,
let 
$$
a_{i}\colon \Sigma\times\sigma(A_i)\longrightarrow\Cdb
$$
be a measurable function such that $a_i(t,\cdotp)\in L^\infty(\lambda_{A_i})$ 
for a.e. $t\in\Sigma$. Then $t\mapsto a_i(t,\cdotp)$ is a $w^*$-measurable function
from $\Sigma$ into $L^\infty(\lambda_{A_i})$, hence $t\mapsto 
\norm{a_i(t,\cdotp)}_{L^\infty(\lambda_{A_i})}$ is measurable for any $i$.
Further by composition (see Remark \ref{WeakTensorisation}), 
$t\mapsto a_i(t,A_i)$ is a $w^*$-measurable function
from $\Sigma$ into $B(\H)$.

\begin{lemma}\label{meas}
Assume that
\begin{equation}\label{assump-P}
\int_{\Sigma}\norm{a_1(t,\cdotp)}_{L^\infty(\lambda_{A_1})}
\norm{a_2(t,\cdotp)}_{L^\infty(\lambda_{A_2})}\cdots
\norm{a_n(t,\cdotp)}_{L^\infty(\lambda_{A_n})}\,\text{d}\mu(t)\,<\infty\,.
\end{equation}
Then for any $X_1,\ldots,X_{n-1}\in S^2(\H)$, the function 
\begin{equation}\label{Fedor}
\Sigma\longrightarrow S^2(\H),\qquad
t\mapsto\, a_1(t,A_1)X_1 a_2(t,A_2)X_2\cdots X_{n-1}a_n(t,A_n),
\end{equation}
is integrable.
\end{lemma}

\begin{proof} 
Fix $X_1,\ldots,X_{n-1}\in S^2(\H)$. 
Write $X_1=X'X''$ with 
$X',X''\in S^4(\H)$, the 4-th order Schatten space on $\H$. By composition, $t\mapsto a_1(t,A_1)X'$
is a $w^*$-measurable function
from $\Sigma$ into $S^4(\H)$. Since $S^4(\H)$ is reflexive and separable,
it follows from \cite[Theorem II.2]{Diestel} that $t\mapsto a_1(t,A_1)X'$
is actually measurable from $\Sigma$ into $S^4(\H)$. 
Likewise $t\mapsto X''a_2(t,A_2)$
is measurable from $\Sigma$ into $S^4(\H)$. Since
$$
a_1(t,A_1)X_1a_2(t,A_2)=(a_1(t,A_1)X')(X''a_2(t,A_2)),
$$ 
it follows that 
$t\mapsto a_1(t,A_1)X_1a_2(t,A_2)$ is measurable from $\Sigma$
into $S^2(\H)$. One proves similarly that
$t\mapsto X_2a_3(t,A_3)\cdots X_{n-1} a_n(t,A_n)$ 
is measurable from $\Sigma$
into $S^2(\H)$. We deduce that the function (\ref{Fedor})
is measurable.

Then the assumption (\ref{assump-P}) ensures that this function
is integrable.
\end{proof}

\begin{proposition}\label{IPTP}
Assume (\ref{assump-P}) and let $\phi\in L^\infty(\lambda_{A_1}\times\cdots\times\lambda_{A_n})$ 
be defined by setting
\begin{equation}\label{phi-P}
\phi(t_1,t_2,\ldots,t_n)= \int_{\Sigma} a_1(t,t_1)a_2(t,t_2)\cdots a_n(t,t_n)\, \text{d}\mu(t)
\end{equation}
for a.e. $(t_1,\ldots,t_n)$ in $\sigma(A_1)\times\cdots\times \sigma(A_n)$. 
Then
\begin{equation}\label{Peller}
\Gamma^{A_1,\ldots,A_n}(\phi)(X_1,\ldots,X_{n-1}) =
\int_{\Sigma} a_1(t,A_1)X_1 a_2(t,A_2)X_2\cdots X_{n-1}a_n(t,A_n)\, \text{d}\mu(t)
\end{equation}
for any $X_1,X_2,\ldots, X_{n-1}$ in $S^2(\H)$.
\end{proposition}

\begin{proof}
We introduce $\widetilde{a_i}\colon \Sigma\to L^\infty(\lambda_{A_i})$ 
by writing $\widetilde{a_i}(t) = a_i(t,\cdotp)$
for any $i=1,\ldots,n$.
Then the function $\widetilde{\phi}\colon \Sigma\to 
L^\infty(\lambda_{A_1}\times\cdots\times\lambda_{A_n})$ defined by 
$$
\widetilde{\phi}(t) = \widetilde{a_1}(t)\otimes \widetilde{a_2}(t)\otimes\cdots
\otimes \widetilde{a_n}(t),\qquad t\in\Sigma,
$$
is $w^*$-measurable. 
Let $\varphi\in L^1(\lambda_{A_1}\times\cdots\times\lambda_{A_n})$. Then for a.e. $t\in\Sigma$, we have
$$
\langle \widetilde{\phi}(t) ,\varphi\rangle =\int a_1(t,t_1)
\cdots a_n(t,t_n)\varphi(t_1,\ldots,t_n)\,\text{d}\lambda_{A_1}(t_1)\cdots 
\text{d}\lambda_{A_n}(t_n)\,.
$$
Hence by  Fubini's Theorem,
$$
\langle\phi,\varphi\rangle =\int_{\Sigma}  \langle \widetilde{\phi}(t) ,\varphi\rangle \,
\text{d}\mu(t)\,.
$$
Fix $X_1,X_2,\ldots, X_{n}$ in $S^2(\H)$. Write $\Gamma=\Gamma^{A_1,\ldots, A_n}$ for convenience. Since
this mapping is 
$w^*$-continuous, there exists a necessarily unique $\varphi\in L^1(\lambda_{A_1}\times\cdots\times\lambda_{A_n})$
such that for any $\psi \in L^\infty(\lambda_{A_1}\times\cdots\times\lambda_{A_n})$, we have
$$
\bigl\langle \Gamma(\psi)(X_1,\ldots,X_{n-1}), X_n\bigr\rangle = \langle\psi,\varphi\rangle.
$$
We shall apply this identity with $\psi=\phi$ first, and then with $\psi=\widetilde{\phi}(t)$. Then we obtain
\begin{align*}
\bigl\langle \Gamma(\phi)(X_1,\ldots,X_{n-1}), X_n\bigr\rangle 
& = \langle\phi,\varphi\rangle\\
& = \int_{\Sigma}  \langle \widetilde{\phi}(t) ,\varphi \rangle \,
\text{d}\mu(t)\\
& = \int_{\Sigma}  \bigl\langle\Gamma(\widetilde{\phi}(t))(X_1,\ldots,X_{n-1}), X_n\bigr\rangle \,
\text{d}\mu(t)\,.
\end{align*}
By the definition 
of $\Gamma$ on elementary tensor products, we have
$$
\Gamma(\widetilde{\phi}(t))(X_1,\ldots,X_{n-1}) = 
a_1(t,A_1)X_1a_2(t,A_2)X_2\cdots X_{n-1}a_n(t,A_n)
$$
for a.e. $t\in\Sigma$. Consequently 
$$
\bigl\langle \Gamma(\phi)(X_1,\ldots,X_{n-1}), X_n\bigr\rangle 
= 
\int_{\Sigma} \bigl\langle a_1(t,A_1)X_1 a_2(t,A_2)X_2\cdots X_{n-1}a_n(t,A_n), X_n
\bigr\rangle \,\text{d}\mu(t)\,.
$$
This shows (\ref{Peller}).
\end{proof}

Following \cite{Peller2006}, the space of all functions 
$\phi\in L^\infty(\lambda_{A_1}\times\cdots\times\lambda_{A_n})$ 
defined  by (\ref{phi-P}) for some $a_1,\ldots,a_n$ satisfying (\ref{assump-P}) is called
the integral projective tensor product of the spaces 
$L^\infty(\lambda_{A_1}),\ldots,L^\infty(\lambda_{A_n})$; this space is denoted by
$$
L^\infty(\lambda_{A_1})\hat{\otimes}_i\cdots\hat{\otimes}_i L^\infty(\lambda_{A_n}).
$$

In \cite{ACDS, Peller2006} the authors define a multiple operator integral mapping
$$
T_\phi\colon B(\H)\times\cdots\times B(\H)\longrightarrow B(\H)
$$
for any $\phi\in L^\infty(\lambda_{A_1})\hat{\otimes}_i\cdots\hat{\otimes}_i L^\infty(\lambda_{A_n})$,
as follows. Let $\phi$ 
be defined by (\ref{phi-P}) for some $a_1,\ldots, a_n$ satisfying (\ref{assump-P}). Then
for any $X_1,\ldots, X_{n-1}$ in $B(\H)$, the operator $T_\phi(X_1,\ldots, X_{n-1})$ is defined by setting
\begin{equation}\label{T-phi}
\text{tr}\bigl(T_\phi(X_1,\ldots, X_{n-1}) Z\bigr)
=\int_\Sigma\text{tr}\bigl(a_1(t,A_1)X_1 a_2(t,A_2)X_2\cdots X_{n-1}a_n(t,A_n) Z\bigr)\, \text{d}\mu(t)
\end{equation}
for any $Z\in S^1(\H)$. Indeed it follows from \cite[Section 4]{ACDS} 
that for any $X_1,\ldots, X_{n-1}$ in $B(\H)$,
the function 
$$
\Sigma\longrightarrow B(\H),\qquad
t\mapsto\, a_1(t,A_1)X_1 a_2(t,A_2)X_2\cdots X_{n-1}a_n(t,A_n),
$$
belongs to $L^1_\sigma(\Sigma; B(\H))$, and hence $t\mapsto \text{tr}
\bigl(a_1(t,A_1)X_1 a_2(t,A_2)X_2\cdots X_{n-1}a_n(t,A_n) Z\bigr)$
is integrable for any $Z\in S^1(\H)$.

Proposition \ref{IPTP} shows that the constructions from the present paper
are compatible with those from \cite{ACDS, Peller2006}.
Namely for any $\phi$ in the integral projective tensor product, the restriction of $T_\phi$
to $S^2(\H)\times\cdots\times S^2(\H)$ coincides with $\Gamma^{A_1,\ldots, A_n}(\phi)$.

We observe that for any $\phi\in 
L^\infty(\lambda_{A_1})\hat{\otimes}_i\cdots\hat{\otimes}_i L^\infty(\lambda_{A_n})$,
the $(n-1)$-linear bounded operator $T_\phi$ is separately $w^*$-continuous. That is, 
for any $1\leq k\leq n-1$ and for any 
$X_1,\ldots, X_{k-1}$, $X_{k+1},\ldots, X_{n-1}$ in $B(\H)$, the linear map
from $B(\H)$ into itself taking any $X_k\in B(\H)$ to
$T_\phi(X_1,\ldots, X_{n-1})$ is $w^*$-continuous. Let us show this for $k=1$,
the other cases being similar. We consider $\phi$ given by (\ref{phi-P}).
We fix $X_2,\ldots, X_{n-1}$ in $B(\H)$ and $Z\in S^1(H)$. We
let $\eta\colon B(\H)\to\Cdb$ be defined by 
$$
\eta(X)= \text{tr}\bigl(T_\phi(X,X_2\ldots, X_{n-1}) Z\bigr),\qquad X\in B(\H),
$$
and we aim at showing that the functional $\eta$ is $w^*$-continuous.
For we consider $\Theta\colon \Sigma\to S^1(\H)$ defined by setting
$$
\Theta(t) = a_2(t,A_2)X_2\cdots X_{n-1}a_n(t, A_n)Z a_1(t,A_1)
$$
for a.e. $t\in\Sigma$.  Arguing as in the proof of Lemma \ref{meas}, one shows that $\Theta$
is measurable and hence that $\Theta$ is integrable. Then it follows from 
(\ref{T-phi}) that for any $X\in B(\H)$, we have
$$
\eta(X) = \text{tr}\Bigl(X\,\int_\Sigma \Theta(t)\,\text{d}\mu(t)\,\Bigr).
$$
This shows that $\eta$ is $w^*$-continuous and concludes the proof.

This leads to the following.

\begin{corollary} For any $\phi$
in the space $L^\infty(\lambda_{A_1})
\hat{\otimes}_i\cdots\hat{\otimes}_i L^\infty(\lambda_{A_n})$, the 
$(n-1)$-linear map $\Gamma^{A_1,\ldots, A_N}(\phi)\colon S^2(\H)\times\cdots\times
S^2(\H)\to S^2(\H)$ 
extends to a (necessarily unique) separately $w^*$-continuous bounded $(n-1)$-linear map
$B(\H)\times\cdots\times B(\H)\longrightarrow B(\H)$.
\end{corollary}

In the case $n=2$, $L^\infty(\lambda_{A})\hat{\otimes}_i L^\infty(\lambda_{B})$
coincides with the space all functions in $L^\infty(\lambda_A\times \lambda_B)$
satisfying condition (iv) from Remark \ref{GT}. Equivalently, we have
$$
\I(L^1(\lambda_{A}), L^\infty(\lambda_{B})) \simeq  
L^\infty(\lambda_{A})\hat{\otimes}_i L^\infty(\lambda_{B}).
$$
The inclusion `$\subset$' is obvious. The non trivial reverse inclusion
is a well-known fact which follows from \cite[Chapter VII, Theorem 9]{Diestel}.
According to this result (see the beginning of Remark \ref{GT}), it suffices to show
that any $\phi\in L^\infty(\lambda_{A})\hat{\otimes}_i L^\infty(\lambda_{B})$
induces a bounded functional on the injective tensor product
$L^1(\lambda_{A})\overset{\vee}{\otimes}  L^1(\lambda_{B})$. To check this
property, consider 
$$
\phi(t_1,t_2)= \int_\Sigma a(t,t_1)b(t,t_2)\,\text{d}\mu(t)
$$
for some measurable functions $a\colon\Sigma\times\sigma(A)\to\Cdb$ and
$b\colon\Sigma\times\sigma(B)\to\Cdb$ 
such that 
$$
K = \int_\Sigma \norm{a(t,\cdotp)}_\infty\norm{b(t,\cdotp)}_\infty
\, \text{d}\mu(t)\,<\infty\,.
$$ 
Then for any finite families
$(f_k)_k$ in $L^1(\lambda_A)$ and $(g_k)_k$ in $L^1(\lambda_B)$, we have
\begin{align*}
\Bigl\vert\Bigl\langle \phi,\sum_k f_k\otimes g_k\Bigr\rangle\Bigr\vert
& \leq \int_\Sigma \Bigl\vert\sum_k\langle a(t,\cdotp),f_k\rangle \langle b(t,\cdotp),g_k\rangle 
\Bigr\vert\, \text{d}\mu(t)\\
&\leq \, K\,\Bignorm{\sum_k f_k\otimes g_k}_{L^1\overset{\vee}{\otimes} L^1},
\end{align*}
which proves the result.

We finally turn to the case $n=3$. Consider three normal operators $A,B,C$ on $\H$.
It is clear that for any $\phi\in L^\infty(\lambda_A)\hat{\otimes}_i L^\infty(\lambda_B)
\hat{\otimes}_i L^\infty(\lambda_C)$, $\Gamma^{A,B,C}(\phi)$ 
extends to a bounded bilinear map $S^2(\H)\times S^2(\H)\to S^1(\H)$. Indeed assume that
$$
\phi(t_1,t_2,t_3)= \int_\Sigma a(t,t_1)b(t,t_2)c(t,t_3)\,\text{d}\mu(t)
$$
for some measurable functions $a\colon\Sigma\times\sigma(A)\to\Cdb$,
$b\colon\Sigma\times\sigma(B)\to\Cdb$ and $c\colon\Sigma\times\sigma(C)\to\Cdb$
such that 
$$
K = \int_\Sigma \norm{a(t,\cdotp)}_\infty\norm{b(t,\cdotp)}_\infty
\norm{c(t,\cdotp)}_\infty\, \text{d}\mu(t)\,<\infty\,.
$$ 
Then for any
$X,Y$ in $S^2(\H)$,
$$
\int_\Sigma\bignorm{a(t,A)Xb(t,B)Yc(t,C)}_{1}\, \text{d}\mu(t)\ \leq
K\norm{X}_2\norm{Y}_2.
$$
Hence by Proposition \ref{IPTP}, $\Gamma^{A,B,C}(\phi)(X,Y)$ belongs to $S^1(\H)$ and
we have 
$$
\norm{\Gamma^{A,B,C}(\phi)(X,Y)}_1\leq K\norm{X}_2\norm{Y}_2,\qquad 
X,Y \in S^2(\H).
$$
Example \ref{Ex1} below shows that the converse is wrong, that is, there 
exist functions $\phi\in L^\infty(\lambda_A\times\lambda_B\times\lambda_C)$
such that $\Gamma^{A,B,C}(\phi)\colon S^2(\H)\times S^2(H)\to S^1(\H)$ 
although $\phi$ does not belong to $L^\infty(\lambda_A)\hat{\otimes}_i L^\infty(\lambda_B)
\hat{\otimes}_i L^\infty(\lambda_C)$. (There are actually a lot of
such functions.)

\begin{example}\label{Ex1} 
In this paragraph, and in Example \ref{Ex2} below, we 
consider families $M=\{m_{ikj}\}_{i,k,j\geq 1}$ in $\ell^{\infty}(\Ndb^3)$ 
to which we associate the bilinear Schur multiplier 
$B_M\colon S^2\times S^2\to S^2$ defined by 
$$
B_M(X,Y)= \Bigl[\sum_{k\geq 1} m_{ikj} x_{ik} y_{kj}\Bigr]_{i,j\geq 1},
\qquad X=[x_{ij}]_{i,j\geq 1},\ Y=[y_{ij}]_{i,j\geq 1} \,\in S^2.
$$
Bilinear maps $B_M$ are special cases of the bilinear maps $\Lambda(\phi)$
and $\Gamma(\phi)$ considered in subsections \ref{Functions} and \ref{Operators}.

Let $S=\{s_{kj}\}_{k,j\geq 1} \in \ell^\infty(\Ndb^2)$ and let 
$L_S\colon S^2\to S^2$ be  the associated
Schur multiplier defined by 
$$
L_S\bigl([y_{kj}]_{k,j\geq 1}\bigr)
= \bigl([s_{kj} y_{kj}]_{k,j\geq 1}\bigr),\qquad 
[y_{kj}]_{k,j\geq 1}\in S^2.
$$
Set $m_{ikj}=s_{kj}$ for any $i,j,k\geq 1$. It follows from
the above definitions that for any $X,Y\in S^2$,
$$
B_M(X,Y) = X L_S(Y).
$$
Since $\norm{L_S(Y)}_2\leq \norm{S}_\infty\norm{Y}_2$,
this implies that $B_M\colon S^2\times S^2\to S^1$ boundedly.
The above formula also implies that $B_M$ extends to a bounded bilinear map
$B(\ell^2)\times B(\ell^2)\to B(\ell^2)$ if and only if
$L_S$ extends to a bounded map $B(\ell^2)\to B(\ell^2)$.
This holds true if and only if $S\in \ell^\infty\hat{\otimes}_i\ell^\infty$.
Thus whenever $S\in \ell^\infty(\Ndb^2)\setminus \ell^\infty\hat{\otimes}_i\ell^\infty$,
$B_M$ is bounded from $S^2\times S^2$ into 
$S^1$ but $B_M$ is not bounded from $B(\ell^2)\times B(\ell^2)$
into $B(\ell^2)$. In this case, $M$ cannot belong to 
$\ell^\infty\hat{\otimes}_i\ell^\infty\hat{\otimes}_i\ell^\infty$.
\end{example}

\begin{example}\label{Ex2}
To complement the above discussion, 
let us show the existence of (plenty of) families
$M\in \ell^{\infty}(\Ndb^3)$ such that 
\begin{itemize}
\item [(i)] $B_M$ extends to a bounded bilinear map $B(\ell^2)\times B(\ell^2)\to B(\ell^2)$;
\item [(ii)] $B_M$ does not extend to a bounded bilinear map $S^2\times S^2\to S^1$.
\end{itemize}

Let $S=\{s_{ij}\}_{i,j\geq 1} \in \ell^\infty(\Ndb^2)$.
Set $m_{i1j} = s_{ij}$ for any $i,j\geq 1$ and,
for any $k\geq 2$, set $m_{ikj}=0$ for any $i,j\geq 1$. 
For any 
$X=[x_{ij}]_{i,j\geq 1}$ and 
$Y=[y_{ij}]_{i,j\geq 1}$ in $S^2$, we have
$$
B_M(X,Y) =\bigl[s_{ij} x_{i1}y_{1j}\bigr]_{i,j\geq 1}.
$$

For any finite families $(\alpha_j)_{j\geq 1}$ and 
$(\beta_i)_{i\geq 1}$ of complex numbers, we have
\begin{align*}
\Bigl\vert \sum_{i,j\geq 1} s_{ij} x_{i1}y_{1j}\alpha_j\beta_i\Bigr\vert
& \leq \norm{S}_\infty \sum_{i,j\geq 1} \vert x_{i1} y_{1j}\alpha_j\beta_i\vert \\
& \leq \norm{S}_\infty \Bigl(\sum_{i\geq 1}\vert x_{i1}\vert^2 \Bigr)^{\frac12}
\Bigl(\sum_{i\geq 1}\vert \alpha_{i}\vert^2\Bigr)^{\frac12}
\Bigl(\sum_{j\geq 1}\vert y_{1j}\vert^2\Bigr)^{\frac12}
\Bigl(\sum_{j\geq 1}\vert \beta_{j}\vert^2\Bigr)^{\frac12}\\
&\leq \norm{S}_\infty \norm{X}_{B(\ell^2)}\norm{Y}_{B(\ell^2)}
\Bigl(\sum_{i\geq 1}\vert \alpha_{i}\vert^2\Bigr)^{\frac12}
\Bigl(\sum_{j\geq 1}\vert \beta_{j}\vert^2\Bigr)^{\frac12},
\end{align*}
by the Cauchy-Schwarz inequality. This shows that $B_M$
satisfies (i).

We now claim that if $B_M$ extends to a bounded bilinear map $S^2\times S^2\to S^1$,
then $L_S$ extends to a bounded map $B(\ell^2)\to B(\ell^2)$. Indeed suppose that
$$
K = \bignorm{B_M\colon S^2\times S^2\longrightarrow S^1}\,<\infty.
$$ 
Consider a 
finite matrix $[z_{ij}]_{i,j\geq 1}$ and finite families
$(\alpha_i)_{i\geq 1}$ and 
$(\beta_j)_{j\geq 1}$ of complex numbers.
Let $X=[x_{ij}]_{i,j\geq 1}$ be defined
by setting $x_{i1}=\alpha_i$ for any $i\geq 1$
and, for any $j\geq 2$, $x_{ij}=0$ for any $i\geq 1$.
Likewise, let $Y=[y_{ij}]_{i,j\geq 1}$
be defined
by setting $y_{1j}=\beta_j$ for any $j\geq 1$
and, for any $i\geq 2$, $y_{ij}=0$ for any $j\geq 1$.
Then
$$
\norm{X}_{2} = \Bigl(\sum_{j\geq 1}\vert \beta_{j}\vert^2\Bigr)^{\frac12}
\qquad\hbox{and}\qquad
\norm{Y}_{2} = \Bigl(\sum_{i\geq 1}\vert \alpha_{i}\vert^2\Bigr)^{\frac12}.
$$
Hence
$$
\bigl\vert \text{tr}\bigl(B_M(X,Y)Z\bigr)\bigr\vert\,\leq\, K\norm{Z}_{B(\ell^2)}
\Bigl(\sum_{j\geq 1}\vert \beta_{j}\vert^2\Bigr)^{\frac12}
\Bigl(\sum_{i\geq 1}\vert \alpha_{i}\vert^2\Bigr)^{\frac12}.
$$
Since 
$$
\text{tr}\bigl(B_M(X,Y)Z\bigr) = \sum_{i,j\geq 1} s_{ij} x_{i1}y_{1j} z_{ji}
= \sum_{i,j\geq 1} s_{ij} \alpha_{i} \beta_{j} z_{ji},
$$
this implies that
$$
\Bigl\vert \sum_{i,j\geq 1} s_{ij} z_{ji}\alpha_{i} \beta_{j} 
\Bigr\vert\,\leq\, K\norm{Z}_{B(\ell^2)}
\Bigl(\sum_{j\geq 1}\vert \beta_{j}\vert^2\Bigr)^{\frac12}
\Bigl(\sum_{i\geq 1}\vert \alpha_{i}\vert^2\Bigr)^{\frac12}.
$$
This implies that $L_S$ extends to a bounded map $B(\ell^2)\to B(\ell^2)$ and proves the claim.

Thus for any $S\in \ell^\infty(\Ndb^2)\setminus \ell^\infty\hat{\otimes}_i\ell^\infty$,
the associated family $M$ satisfies (ii).
\end{example}

We finally refer to \cite{ER0,JTT} for the study of multilinear measurable
Schur multipliers which extend to completely bounded maps
$B(L^2)\times\cdots\times B(L^2)\to B(L^2)$.

\vskip 1cm
\noindent
{\bf Added, April 2020:} After a first version of this paper was circulated 
in 2017, some of its results have been used in \cite{C, C1, CLSS, LS, ST}.

\bigskip
\noindent
{\bf Acknowledgements.} 
The first two authors were supported by the French 
``Investissements d'Avenir" program, 
project ISITE-BFC (contract ANR-15-IDEX-03). The third
author was partially supported by the ARC (grant FL170100052 ).

\vskip 1cm

\end{document}